\pdfoutput=1
\documentclass[12pt,twoside,a4paper]{amsart}

\usepackage{amssymb}
\usepackage{amsmath,amscd}
\usepackage{epic}
\usepackage{bm}
\usepackage[all]{xy}
\usepackage[initials,nobysame,shortalphabetic,lite]{amsrefs}

\theoremstyle{plain}
\newtheorem{thm}{Theorem}[section]
\newtheorem{lma}[thm]{Lemma}
\newtheorem{cor}[thm]{Corollary}

\theoremstyle{definition}
\newtheorem{defi}[thm]{Definition}

\begin{document}

\title{Tight maps and holomorphicity}

\author{Oskar Hamlet}

\address{Department of Mathematics\\Chalmers University of Technology and the University of Gothenburg\\412 96 G\"OTEBORG\\SWEDEN}

\email{hamlet@chalmers.se}

\begin{abstract}
Tight maps were introduced and studied along tight homomorphisms by Burger, Iozzi and Wienhard with aims towards maximal representations. In this paper we show that, with the exception of maps from the Poincar\'e disc, tight maps into classical Hermitian symmetric spaces must be holomorphic or antiholomorphic. Together with previous results this completely classifies tight maps into classical codomains.
\end{abstract}

\maketitle


\section{Introduction}
Let $(\mathcal{X}_i,\omega_i)$, $i=1,2$, be Hermitian symmetric spaces of noncompact type paired with some choice of invariant K\"ahler forms. A map $\rho\colon\mathcal{X}_1\rightarrow\mathcal{X}_2$ is called \emph{totally geodesic} if the image of every geodesic in $\mathcal{X}_1$ is a geodesic in $\mathcal{X}_2$. A totally geodesic map $\rho\colon \mathcal{X}_ 1\rightarrow \mathcal{X}_2$ satisfies 
\begin{equation}\label{t2}
sup_{\Delta\in\mathcal{X}_1}\int_\Delta{\rho^*\omega_2}\leq sup_{\Delta\in\mathcal{X}_2}\int_\Delta{\omega_2}
\end{equation}
 where the supremum is taken over triangles with geodesic sides. We say that the map is \emph{tight} if equality holds in (\ref{t2}). 
There is  also a paralell notion of tightness for group homomorphisms when the codomain is a Hermitian Lie group.
We say that the homomorphism is tight if the norm of the bounded K\"ahler class is preserved under the pullback induced by the homomorphism.

The motivation for the study of tight maps and homomorphisms comes from a structure theorem of maximal representations due to Burger, Iozzi and Wienhard \cite{B9}. It states that the action on a Hermitian symmetric space $\mathcal{X}$ coming from a maximal representation leaves invariant a tightly embedded subspace $\mathcal{Y}\subset\mathcal{X}$.
It is a fundamental question to determine situations in which such a map is holomorphic or antiholomorphic.
Tight homomorphisms were also an important tool in the work of Kim and Pansu \cite{B2} where they investigated when a surface group representation could be approximated by Zariski dense representations.

Tight maps were introduced and extensively studied in \cite{B8}. In the paper they also classified all tight maps from the Poincar\'e disc. Among these they found both holomorphic and nonholomorphic tight maps; however, they were unable to find nonholomorphic tight maps from higher dimensional spaces. They asked if tight maps should always be (anti-) holomorphic in the higher dimensional case. In this paper we confirm this in the classical case.
\begin{thm}\label{main}
Let $\mathcal{X}_ 1$ and $\mathcal{X}_2$ be irreducible Hermitian symmetric spaces. Assume that $\mathcal{X}_ 1$ is not the Poincar\'e disc and that $\mathcal{X}_ 2$ is classical.
If $\rho\colon \mathcal{X}_ 1\rightarrow \mathcal{X}_2$ is a tight map, then it is (anti-) holomorphic. 
\end{thm}
We also get a partial result for the exceptional Hermitian symmetric spaces.
\begin{thm}\label{bonus}
Let $\mathcal{X}'$ be the exceptional Hermitian symmetric space associated to the symmetric pair $(\mathfrak{e}_{6(-14)},\mathfrak{so}(10)+\mathbb{R})$. Further let $\mathcal{X}$ be an irreducible Hermitian symmetric space of rank at least two. 
If $\rho\colon \mathcal{X}\rightarrow \mathcal{X}'$ is a tight map, then it is (anti-) holomorphic. 
\end{thm}
We have recently proved that there exists no nonholomorphic tight maps in the remaining exceptional cases \cite{B13}.
The results in \cite{B13} does not supersede those of this paper but rather uses ad hoc techniques for the remaining exceptional cases.

The results in this paper together with the results in \cite{B8}, \cite{B5} and \cite{B13} thus yield a complete classification of tight maps from irreducible Hermitian symmetric spaces.

\subsection{Outline of the paper and the proof}
As our proof is rather technical we outline here the main ideas and the structure of the paper. 

We start in section 2 by setting notation and discussing equivalent formulations. We will mainly approach the problem from the perspective of Lie algebra homomorphisms.

In section 3 we begin by giving a brief introduction to continuous bounded cohomology. We use this to investigate when the composition and the product of maps are tight. Roughly one could say that the composition (or product) of two maps is tight if and only if the individual maps are both tight. 

In section 4 we recall the basics from representation theory of semisimple Lie algebras. Representation theory is the study of homomorphisms $\rho'\colon\mathfrak{g}\rightarrow\mathfrak{gl}(n,\mathbb{C})$; we will use it to better understand homomorphisms $\rho\colon\mathfrak{g}\rightarrow\mathfrak{su}(p,q)$. We do this by considering $\rho$ as a homomorphism 
$\rho\colon\mathfrak{g}\rightarrow\mathfrak{gl}(p+q,\mathbb{C})$ whose image is contained in a subalgebra $\mathfrak{su}(p,q)\subset\mathfrak{gl}(,p+q,\mathbb{C})$. This causes some technical issues which we address in section 4.

In section 5 we combine what we have gathered so far into a new criterion for non-tightness in a limited setting. This criterion captures one of the main ideas of the paper. The broad strokes of the argument go as follows. 
Suppose that $\rho\colon\mathfrak{g}\rightarrow\mathfrak{su}(p,q)$ is an irreducible representation and $\mathfrak{g}_0\subset\mathfrak{g}$ a tightly embedded subalgebra. Restricting $\rho$ to $\mathfrak{g}_0$ we get a new representation $\rho|\colon\mathfrak{g}_0\rightarrow\mathfrak{su}(p,q)$ which no longer is irreducible. We can thus write $\rho|$ as a sum of irreducible representations, i.e. $\rho|=\sum_{i=1}^n\rho^i$. From section 4 we know that $\rho|$ then factors as $\iota\circ (\rho^1,...,\rho^n)\colon\mathfrak{g}_0\rightarrow\oplus_{i=1}^n\mathfrak{su}(p_i,q_i)\rightarrow\mathfrak{su}(p,q)$ with $\iota$ a  holomorphic embedding. Suppose now that one $\rho^i$ is nontight. By what we know about compositions and products this implies that $(\rho^1,...,\rho^n)$ is nontight. This in turn implies that the composition $\iota\circ(\rho^1,...,\rho^n)=\rho|$ is nontight. But that $\rho|$ is nontight implies that $\rho$ is nontight since the inclusion $\mathfrak{g}_0\subset\mathfrak{g}$ was assumed to be tight. We can thus show nontightness of a representation $\rho\colon\mathfrak{g}\rightarrow\mathfrak{su}(p,q)$ by considering the branching of $\rho$ when restricted to a tightly embedded subalgebra. 

In section 6 we begin by calculating which representations \\*$\rho\colon\mathfrak{su}(1,1)^{\oplus i}\rightarrow\mathfrak{su}(p,q)$ are tight for $i=1,2$. We then combine this knowledge with the criterion from section 5 to show, via calculation, that there are no tight nonholomorphic homomorphisms from $\mathfrak{sp}(4,\mathbb{R})$ or $\mathfrak{su}(2,1)$\ into $\mathfrak{su}(p,q)$. We also show that any tight homomorphism $\rho\colon\mathfrak{sp}(4,\mathbb{R})\oplus\mathfrak{su}(1,1)\rightarrow\mathfrak{su}(p,q)$ is holomorphic or antiholomorphic when restricted to $\mathfrak{sp}(4,\mathbb{R})$.

In section 7 we address the general case. We divide it into cases dependending on the rank of the domain. Suppose that $\rho\colon\mathfrak{g}_1\rightarrow\mathfrak{g}_2$ is tight and nonholomorphic.

If $\mathfrak{g}_1$ is of even real rank there is, by the classification of holomorphic tight maps in \cite{B5}, a tight and holomorphic embedding $\iota\colon\mathfrak{sp}(4,\mathbb{R})\rightarrow\mathfrak{g}_1$. The composition $\rho\circ\iota$ is then a tight and nonholomorphic homomorphism. Since $\mathfrak{sp}(4,\mathbb{R})$ is of tube type we know by a structure theorem in \cite{B8} that the image of $\rho\circ\iota$ is contained in a tightly and holomorphically embedded subalgebra of tube type $\mathfrak{g}_2^T\subset\mathfrak{g}_2$. Restricting the codomain (and slightly abusing the notation) we have a 
tight and nonholomorphic homomorphism $\rho\circ\iota\colon\mathfrak{sp}(4,\mathbb{R})\rightarrow\mathfrak{g}_2^T$. By the classification of holomorphic tight maps we know that there is a tight and holomorphic homomorphism $\iota'\colon\mathfrak{g}_2^T\rightarrow\mathfrak{su}(n,n)$ for some $n$. The composition $\iota'\circ\rho\circ\iota$ is thus a tight and nonholomorphic homomorphism from $\mathfrak{sp}(4,\mathbb{R})$ into $\mathfrak{su}(n,n)$. But this contradicts our results from section 5, hence $\rho$ can not be tight and nonholomorphic.

If $\mathfrak{g}_1$ is of odd rank greater than one we can argue in a similar fashion with $\mathfrak{sp}(4,\mathbb{R})$ replaced by $\mathfrak{sp}(4,\mathbb{R})\oplus\mathfrak{su}(1,1)$.

Finally if $\mathfrak{g}_1$ is of rank one we can embed $\mathfrak{su}(2,1)$ into $\mathfrak{g}_1$ with a tight and holomorphic homomorphism $\iota$. If $\mathfrak{g}_2$ is of tube type we can compose with a holomorphic tight embedding $\iota'\colon\mathfrak{g}_2\rightarrow\mathfrak{su}(n,n)$ and argue as before. If $\mathfrak{g}_2$ is not of tube type we can not use the structure theorem since $\mathfrak{su}(2,1)$ is not of tube type. We treat this case separately in section 7.

\section{Preliminaries}\label{prel}

We start by setting the notation that will be used throughout the paper.
We denote by $\mathcal{X}$ Hermitian symmetric spaces of noncompact type, by $G$ the identity component of the isometry group of $\mathcal{X}$, by $K$ the stabilizer of a chosen basepoint $0$ and by $\mathfrak{g}$ the Lie algebra of $G$ with Cartan decomposition $\mathfrak{g}=\mathfrak{k}+\mathfrak{p}$. We identify $T_0\mathcal{X}\simeq\mathfrak{p}$ and we use the letters $X,Y$ to denote either tangent vectors or elements of $\mathfrak{g}$. Further we use brackets $\langle\cdot,\cdot\rangle$ to denote the invariant Riemannian metric on $\mathcal{X}$, normalized such that the holomorphic sectional curvature is $-1$, as well as the Killing form of $\mathfrak{g}$, it should be clear from the context which is meant. The invariant complex structure on $\mathcal{X}$ is denoted by $J$ and the element in the center of $\mathfrak{k}$ inducing the complex structure on $\mathfrak{p}$ by $Z$. Finally, we denote by $\omega$ the associated K\"ahler form defined by $\omega(X,Y)=\langle JX,Y\rangle$.
By the indexation it should be clear which spaces, groups etc. belong together.
We say that a Lie group is \emph{Hermitian} if it is the identity component of an isometry group of a Hermitian symmetric space of noncompact type or a finite covering group of such. We will use the term \emph{nonholomorphic} to mean neither holomorphic nor antiholomorphic.

With some notation in place we proceed to define tight maps. A map $\rho\colon\mathcal{X}_1\rightarrow\mathcal{X}_2$ is said to be \emph{totally geodesic} if the image of every geodesic in $\mathcal{X}_1$ is a geodesic in $\mathcal{X}_2$, possibly not parametrized by arclength. A totally geodesic map $\rho\colon \mathcal{X}_ 1\rightarrow \mathcal{X}_2$ satisfies 
\begin{equation*}
\mbox{sup}_{\Delta\in\mathcal{X}_1}\int_\Delta{\rho^*\omega_2}\leq \mbox{sup}_{\Delta\in\mathcal{X}_2}\int_\Delta{\omega_2}
\end{equation*}
 where the supremum is taken over triangles with geodesic sides. We say that the map is \emph{tight} if equality holds.

Tight maps are studied from three perspectives. Each totally geodesic map has a corresponding Lie algebra homomorphism and a corresponding Lie group homomorphism. We denote all three of these by the same letter. It should be clear from the context which is meant. To avoid overuse of phrases like "... homomorphism corresponding to a tight holomorphic map..." we will frequently attribute properties of Lie group homomorphisms and totally geodesic maps to the corresponding Lie algebra homomorphism.

We say that two totally geodesic maps $\rho,\eta \colon \mathcal{X}_1\rightarrow \mathcal{X}_2$ 
are equivalent if there is a $g\in G_2$ such that $\rho=g\circ\eta$. Since equivalent maps only differ by a holomorphic isometry we see immediately that the notion of tightness is well defined on equivalence classes of maps. We also have corresponding notions of equivalence for Lie group and Lie algebra homomorphisms. We say that two homomorphisms $\rho,\eta\colon\mathfrak{g}_1\rightarrow\mathfrak{g}_2$ are equivalent, or sometimes \emph{equivalent as homomorphisms}, if $\rho(\cdot)=\mbox{Ad}(g)\eta(\cdot)$ for some $g\in G_2$. Two Lie group homomorphisms $\rho,\eta\colon G_1\rightarrow G_2$ are equivalent if $\rho(\cdot)=\mbox{Ad}(g)\eta(\cdot)$ for some $g\in G_2$.

Finally there is another notion of equivalence that will be used which does \emph{not} agree with the others. We will say that two homomorphisms $\rho,\eta\colon\mathfrak{g}\rightarrow\mathfrak{su}(p,q)$ are \emph{equivalent as representations} if there is a $g\in GL(p+q,\mathbb{C})$ such that $\rho(\cdot)=\mbox{Ad}(g)\eta(\cdot)$.

\section{Continuous bounded cohomology}
In this section we recall some of the theory of continuous bounded cohomology. We will use this to answer questions concerning when the composition of two maps is tight. For a thorough review of the theory see \cite{B10}.

Let $G$ be a locally compact second countable group and define $C^k(G, \mathbb{R}):=\{f\colon G^{k+1}\rightarrow \mathbb{R}, f\mbox{ is continuous and bounded}\}$. $C^k(G,\mathbb{R})$ is naturally equipped with the supremum norm. We define a $G$-action on $C^k(G,\mathbb{R})$ as follows $(g\cdot f)(g_0,...,g_k):=f(g^{-1}g_0,...,g^{-1}g_k)$. Denote by $C^k(G,\mathbb{R})^G$ the $G$-invariant elements of $C^k(G,\mathbb{R})$. These form a complex 
$$0\rightarrow C_{cb}^0(G,\mathbb{R})^G\rightarrow^{d_0} C_{cb}^1(G,\mathbb{R})^G\rightarrow^{d_1} C_{cb}^2(G,\mathbb{R})^G\rightarrow^{d_2} ...$$
where $d_{k-1}f(g_0,...,g_k):=\sum_{j=0}^k{(-1)^jf(g_0,...,\hat{g_j},...,g_k)}$. We define 
$$H^k_{cb}(G,\mathbb{R})=\mbox{Ker}(d_k)/\mbox{Im}(d_{k-1}).$$
The norm on $C_{cb}^k(G,\mathbb{R})$ induces a seminorm on $H^k_{cb}(G,\mathbb{R})$ by $$||[f]||:=\mbox{inf}_{h\in[f]}||h||.$$
For locally compact groups this is a norm in degree two \cite{B11}. A homomorphism $\rho\colon G\rightarrow H$ between groups induces a pullback map between the cohomology groups $\rho^*\colon H^k_{cb}(H,\mathbb{R})\rightarrow H^k_{cb}(G,\mathbb{R})$, $[f]\mapsto[f\circ \rho]$. From the definition we see that this must always be norm decreasing.
 
We will from here on restrict our attention to cohomology in degree two and to Hermitian Lie groups.
Let $G$ be a Hermitian Lie group and $\mathcal{X}$ the associated symmetric space. Then $$c_\omega(g_0,g_1,g_2):=\int_{\Delta(g_0 x_0,g_1 x_0,g_2 x_0)}{\omega}$$ defines a cocycle, where $\Delta(g_0 x_0,g_1 x_0,g_2 x_0)$ denotes a triangle\footnote{How we fill $\Delta(g_0 x_0,g_1 x_0,g_2 x_0)$ is not important as $\omega$ is exact.} with geodesic sides with vertices $g_i x_0$ for some point $x_0\in \mathcal{X}$. This cocycle is continuous and bounded.
The corresponding cohomology class is called the \emph{K\"ahler class} and will be denoted by $\kappa_G$\footnote{The more common notation in the literature is $\kappa_G^b$.}. It is implicit in the definition that $\kappa_G$ depends on the complex structure $J$. We will sometimes use the notation $\kappa_{G,J}$ or write $(G,J)$ for a group with a certain complex structure associated to it when the dependence is crucial.
From the definition of this class we see immediatly that $\kappa_{G,-J}=-\kappa_{G,J}$.

The norms $||\kappa_{G}||$ were computed for the classical case in \cite{B4} and equals $r_{G}\pi$, where $r_{G}$ is the real rank of the group $G$. Another approach using the Maslov index in \cite{B14} covered the exceptional cases.
\begin{defi}
Let $\rho\colon G_1\rightarrow G_2$ be a homomorphism between Hermitian Lie groups. We say that $\rho$ is tight if $||\rho^* \kappa_{G_2}||=||\kappa_{G_2}||$.
\end{defi}

The following theorem from \cite[Proposition 6]{B8} allows us to translate results concerning tight homomorphisms to tight maps.
\begin{thm}
The homomorphism $\rho\colon G_1\rightarrow G_2$ is tight if and only if the corresponding  totally geodesic map $\rho\colon\mathcal{X}_1\rightarrow \mathcal{X}_2$ is tight. 
\end{thm}

Let $G=G_1\times...\times G_n$ be a decomposition of $G$ into simple factors and $\mathcal{X}=\mathcal{X}_1\times... \times\mathcal{X}_n$ be the corresponding decomposition of the symmetric space into irreducible symmetric spaces.
The complex structure $J$ on $\mathcal{X}$ determines a complex structure $J_i$ on each $\mathcal{X}_i$. Recall that for an irreducible Hermitian symmetric space there are two possible choices of complex structure. We have, \cite{B8}, 

\begin{equation*}
H_{cb}^2(G)\cong \prod{H^2_{cb}(G_i)}\cong \prod{\mathbb{R}\kappa_{G_i}}
\end{equation*}
and with a slight abuse of notation we write  
$\kappa_{G,J}=\sum_i{\kappa_{G_i,J_i}}$.
The complex structure $J$ for $\mathcal{X}$ thus defines an orientation for each $H^2_{cb}(G_i)$. 

\begin{defi}
We say that a class $\alpha \in H^2_{cb}(G)$ is
\begin{enumerate}
       \item  positive if $\alpha=\sum{\mu_i \kappa_{G_i,J_i}}$ where $\mu_i\geq 0$ for all $i=1,...,n$, and
       \item  strictly positive if $\alpha=\sum{\mu_i \kappa_{G_i,J_i}}$ where $\mu_i> 0$ for all $i=1,...,n$.
   \end{enumerate}
\end{defi}
\begin{defi}\label{poshomo}
We say that a homomorphism $\rho\colon G_1\rightarrow G_2$ is (strictly) positive if $\rho^* \kappa_{G_2}$ is (strictly) positive.
\end{defi}
In analogous way we also define (strictly) negative classes and homomorphisms.
\begin{lma}\label{lma55}
If $\rho\colon G_1\rightarrow G_2$ corresponds to a holomorphic map then it is positive. If it corresponds to a holomorphic and injective map then it is strictly positive.
\end{lma}
\begin{proof}
We first consider the case where both $G_1$ and $G_2$ are simple.
Consider the corresponding totally geodesic map $\rho\colon\mathcal{X}_1\rightarrow\mathcal{X}_2$. If $\rho$ is not injective it is a constant map, we then get $\rho^*\kappa_{G_2}=0$ which shows that $\rho$ is positive but not strictly positive. If $\rho$ is injective the restriction of the Riemannian metric satisfies $\langle\cdot,\cdot\rangle_2|_{\rho(\mathcal{X}_1)}=c_\rho\langle\cdot,\cdot\rangle_1$ for some $c_\rho>0$.
We get
\begin{eqnarray*}
\rho^{*}\omega_2(X,Y)=\langle\rho_{*} X, J_2\rho_{*} Y\rangle_2=\langle\rho_{*} X, \rho_{*}J_1 Y\rangle_2=c_{\rho}\langle X, J_1 Y\rangle_1=c_{\rho}\omega_1 (X,Y)
\end{eqnarray*}
Since $\rho^*\omega_2=c_\rho\omega_1$ we get $\rho^*\kappa_{G_2}=c_\rho\kappa_{G_1}$ due to the naturality of the Dupont isomorphism, i.e. $\rho$ is strictly positive.

We now consider the general case. Let
$\mathcal{X}_1=\mathcal{X}_{1,1}\times ... \times\mathcal{X}_{1,n}$ and 
$\mathcal{X}_2=\mathcal{X}_{2,1}\times ... \times\mathcal{X}_{2,m}$
be decompositions into irreducible symmetric spaces. Let $\iota_i\colon\mathcal{X}_{1,i}\rightarrow\mathcal{X}_1$ be the inclusion map and $\pi_j\colon\mathcal{X}_2\rightarrow\mathcal{X}_{2,j}$ projection onto the $j$:th factor. Define the maps
$\rho_{\cdot,i}=\rho\circ\iota_i$, $\rho_{j,\cdot}=\pi_j\circ\rho$ and $\rho_{j,i}=\pi_j\circ\rho\circ\iota_i$.
That $\rho$ is injective is equivalent to that $\rho_{\cdot,i}$ is injective for all $i=1,...,n$. In turn, $\rho_{\cdot,i}$ is injective if and only if there exists a $j$ such that $\rho_{j,i}$ is injective.
That $\rho$ is holomorphic is equivalent to that $\rho_{j,i}$ is holomorphic for all $i,j$.
With a slight abuse of notation we get 
\begin{eqnarray*}
\rho^*\kappa_{G_2}=\sum_j{\rho_{j,\cdot}^*\kappa_{G_{2,j}}}=\sum_{j,i}{\rho_{j,i}^*\kappa_{G_{2,j}}}=\sum_{j,i}{c_{\rho_{j,i}}\kappa_{G_{1,i}}}=\sum_i{ (\sum_j{c_{\rho_{j,i}}})\kappa_{G_{1,i}}}.
\end{eqnarray*}
Since all $\rho_{j,i}$ are holomorphic we get that all $c_{\rho_{j,i}}$ are greater than or equal to zero, i.e. $\rho$ is positive. If we further assume that $\rho$ is injective we have that at least one $\rho_{j,i}$ is injective for every $i$. Hence at least one $c_{\rho_{j,i}}$ is strictly positive in the sum $\sum_j{c_{\rho_{j,i}}}$ for every $i$, i.e. $\rho$ is strictly positive. 
 \end{proof}

We will now investigate when compositions of homomorphisms and hence of maps are tight. We start with two lemmas from \cite[Lemma 4.9 and 4.10]{B8}.

\begin{lma}\label{factor1}
Let $\rho\colon G_1\rightarrow G_2$ and $\eta\colon G_2\rightarrow G_3$ be homomorphisms. Assume $\rho$ is tight. If $\eta$ is tight and positive or tight and negative then $\eta\circ \rho$ is tight.
\end{lma}
\begin{lma}\label{positive}
Let $\rho\colon G_1\rightarrow G_2$ be a tight homomorphism. Then there exists a complex structure for $\mathcal{X}_1$ such that $\rho$ is tight and positive.
\end{lma}
The following lemma will be very useful.
\begin{lma}\label{factor12}
Let $\rho\colon G_1\rightarrow G_2$ and $\eta\colon G_2\rightarrow G_3$ be homomorphisms. Assume $G_2$ is simple. Then $\eta\circ \rho$ is tight if and only if both $\rho$ and $\eta$ are tight.
\end{lma}
\begin{proof}
Assume $\eta\circ \rho$ is tight. We have
$$||\kappa_{G_3}||=||\rho^*\eta^*\kappa_{G_3}||\leq ||\eta^*\kappa_{G_3}||\leq||\kappa_{G_3}||.$$
Hence $||\eta^*\kappa_{G_3}||=||\kappa_{G_3}||$ i.e. $\eta$ is tight. Since $G_2$ is simple we have $H^2_{cb}(G_2)=\mathbb{R}\kappa_{G_2}$ and hence $\eta^*\kappa_{G_3}=\pm\frac{r_{G_3}}{r_{G_2}}\kappa_{G_2}$.
We have 
$$r_{G_3}\pi=||\kappa_{G_3}||=||\rho^*\eta^*\kappa_{G_3}||=||\pm \rho^*\frac{r_{G_3}}{r_{G_2}}\kappa_{G_2}||=\frac{r_{G_3}}{r_{G_2}}||\rho^*\kappa_{G_2}||$$
Hence $||\rho^*\kappa_{G_2}||=r_{G_2}\pi=||\kappa_{ G_2}||$, i.e. $\rho$ is tight.

Assume that $\rho$ and $\eta$ are tight. Since $G_2$ is simple there are only two associated complex strutures to it. Hence $\eta$ is tight and positive or negative. Lemma \ref{factor1} then implies that $\eta\circ \rho$ is tight.
 \end{proof}

What we want next is some kind of contrapositive of Lemma \ref{factor1}, that $\eta\circ \rho$ not tight implies $\eta$ not tight. 
To achieve this we will have to vary the complex structure for the symmetric space associated to the middle group. Let $\mathcal{J}_2$ denote the set of complex structures of $\mathcal{X}_2$ and $\mathcal{J}_2'$ a minimal subset of $\mathcal{J}_2$ fulfilling $\mathcal{J}_2'\cup -\mathcal{J}_2'=\mathcal{J}_2$. 
\begin{lma}\label{factor11}
Let $\eta\colon G_2\rightarrow G_3$ be a fixed homomorphism. Further, for every $J\in\mathcal{J}_2'$ let $\rho^J\colon G_1\rightarrow (G_2,J)$ be a homomorphism that is tight with respect to the complex structure $J$. If $\eta\circ \rho^J$ is nontight for all $J\in\mathcal{J}_2'$ then $\eta$ is nontight. 
\end{lma}
\begin{proof}
We prove the contrapositive of the statement. If $\eta$ is tight, there is by Lemma \ref{positive} a complex structure $J$ on $\mathcal{X}_2$ such that $\eta$ is tight and positive. Either $J$ or $-J$ is in $\mathcal{J}_2'$. In the first case $\eta$ is tight and positive hence $\eta\circ \rho^J$ is tight by Lemma \ref{factor1}. In the second case $\eta$ is tight and negative and again $\eta\circ \rho^{-J}$ is tight by Lemma \ref{factor1}.
 \end{proof}

We have the following lemma which is an easy generalization of \cite[Lemma 3.1]{B5}. 
\begin{lma}\label{factor2}
Let $\rho\colon G\rightarrow \prod_{i=1}^{n}{G_i}$ be a homomorphism corresponding to an injective totally geodesic map, where all the $G_j$ are simple. Let $\pi_j\colon\prod_{i=1}^{n}{G_i}\rightarrow G_j$ be the projection maps for $j=1,...,n$.
Then $\rho$ is tight if and only if $\pi_j\circ \rho$ is tight and positive for all $j$ or tight and negative for all $j$. 
\end{lma}
Again we are more interested in when $\rho$ fails to be tight. For this it suffices that one single $\pi_j\circ f$ fails to be tight.
\begin{lma}\label{factor3}
Let $\rho\colon G\rightarrow H$ and $\eta\colon H\rightarrow L$ be homomorphisms of Hermitian Lie groups. Assume that $\eta$ is strictly positive. If $\rho$ is nontight then $\eta\circ \rho$ is nontight.
\end{lma}
\begin{proof}
Let $H=\prod_{i=1}^n{H_i}$ and $G=\prod_{j=1}^N{G_j}$ be decompositions of $H$ and $G$ into simple factors. Denote by $\kappa_L, \kappa_{H_i},\kappa_{G_j}$ the K\"ahler classes. We have that $\eta^*\kappa_L=\sum_{i=1}^n{\lambda_i \kappa_{H_i}}$ where $\sum_{i=1}^n{\lambda_i r_{H_i}}\leq r_{L}$ and all $\lambda_i>0$ since $\eta$ is strictly positive.
We have $\rho^*\eta^*\kappa_L=\sum_{i,j}{\lambda_i\mu_{ij}\kappa_{G_j}}$ where $\sum_{j}{|\mu_{ij}|r_{G_j}}\leq r_{H_i}$.
Now if $\rho$ is not tight, we have a strict inequality for some $i$. Thus we get
$||\rho^*\eta^*\kappa_L||=\sum_{i,j}{||\lambda_i\mu_{ij}\kappa_{G_j}||}=\sum_{i,j}{\lambda_i|\mu_{ij}|r_{G_j}\pi}<\sum_{i}{\lambda_i r_{H_i}\pi}\leq r_L\pi=||\kappa_L||$, 
i.e. $\rho\circ \eta$ is not tight.
 \end{proof}

\section{Representation theory}
In this section we will recall some facts from representation theory that will be needed in Section 5. As there are quite a lot of different notions we start with a subsection going through the basics while setting notation that will be used throughout the remainder of the paper.
\subsection{Notation}
Let $\mathfrak{g}$ be a complex semisimple Lie algebra of rank $n$ with root space decomposition $\mathfrak{g}=\mathfrak{h}+\sum_{\alpha\in A}{\mathfrak{g}_\alpha}$. Here $\mathfrak{h}$ is a \emph{Cartan subalgebra}, $A\subset \mathfrak{h}^*$ the set of \emph{roots} and $\mathfrak{g}_\alpha$ subspaces of $\mathfrak{g}$ such that $[H,X]=\alpha(H)X$ for all $H\in\mathfrak{h}$, $X\in\mathfrak{g}_\alpha$. There is a subset $\{\alpha_1,...,\alpha_n\}\subset A$ called \emph{simple roots} defined by the property that any $\alpha\in A$ can be written as $\alpha=\sum_{i=1}^n{a_i\alpha_i}$ with either all $a_i$ nonpositive or all $a_i$ nonnegative integers. We get a partial ordering on the set of roots by saying that $\alpha=\sum_{i=1}^n{a_i\alpha_i}\geq 0$ if the $a_i\geq 0$, and that $\alpha\geq\beta$ if $\alpha- \beta\geq 0$. Using the Killing form we define $H^\alpha$ by $\langle H^\alpha,H\rangle=\alpha (H)$ for all $H\in\mathfrak{h}$. The \emph{coroots} are the elements $H_\alpha=2\frac{H^\alpha}{\langle H^\alpha,H^\alpha\rangle}$.

A complex representation of $\mathfrak{g}$ is a homomorphism $\rho\colon\mathfrak{g}\rightarrow \mathfrak{gl}(V)$ where $V$ is a complex vector space. We will often refer to $V$ as the representation. 
We say that $\rho\colon\mathfrak{g}\rightarrow \mathfrak{gl}(V)$ and $\rho'\colon\mathfrak{g}\rightarrow \mathfrak{gl}(V')$ are equivalent if there exists a vector space isomorphism $\theta\colon V\rightarrow V'$ such that $\theta^{-1}\circ\rho'(X)\circ\theta=\rho(X)$ for all $X\in\mathfrak{g}$. A representation is said to be \emph{irreducible} if the only $\rho(\mathfrak{g})$-invariant subspaces of $V $ are $\{0\}$ and $V$ itself. An arbitrary representation $V$ decomposes into a sum of irreducible ones, though the decomposition is not necessarily unique. The equivalence classes of irreducible representations appearing however, are unique. A \emph{weight} $\omega$ is an element of $\mathfrak{h}^*$ paired with a subspace $V_\omega\subset V$ such that $\rho(H)v=\omega(H)v$ for all $v\in V_\omega$ and $H\in\mathfrak{h}$. The vector $v$ is called a \emph{weight vector}. The partial ordering of the roots gives us a partial ordering of the weights of a representation. An irreducible representation is determined up to equivalence by its \emph{highest weight}. There is a set $\{\omega_1,...,\omega_n\}\subset\mathfrak{h}^*$ called \emph{fundamental weights} defined as the dual base of the simple coroots, i.e. $\omega_i(H_{\alpha_j})=\delta_{ij}$. Each weight can be written as a sum of fundamental weights with rational coefficients. We denote such weights as $\omega_{m_1,...,m_n}=\sum_{i=1}^n{m_i\omega_i}$ and a corresponding weight vector by $v_{m_1,...,m_n}$. We denote an irreducible representation with highest weight $\omega_{m_1,...,m_n}$ by $\rho_{m_1,...,m_n}^{\mathbb{C}}$. We will sometimes refer to it as a representation with highest weight $(m_1,...,m_n)$.
The \emph{Weyl group} $\mathcal{W}$ of $\mathfrak{g}$ acts on $\mathfrak{h}^*$ and is generated by the reflections $\beta\mapsto\beta - 2\frac{\langle\beta,\alpha\rangle}{\langle\alpha,\alpha\rangle}\alpha$, $\alpha\in A$. To see what weights appear in a representation $V$ with highest weight $\omega$ one starts by considering the set $\mathcal{W}\cdot \omega$. These points are the corners of a convex set $\mathcal{C}\subset\mathfrak{h}^*$. The weights appearing are those $\omega+\sum{a_i\alpha_i}$, $a_i\in\mathbb{Z}$, that lie in $\mathcal{C}$.
Tensor products will be important when realizing representations.
Let $\rho^1\colon\mathfrak{g}_1\rightarrow \mathfrak{gl}(V)$ and $\rho^2\colon\mathfrak{g}_2\rightarrow \mathfrak{gl}(U)$ be representations. We will denote by $\rho^1\boxtimes \rho^2\colon\mathfrak{g}_1\oplus\mathfrak{g}_2\rightarrow \mathfrak{gl}(V\otimes U)$ the representation defined by 
$$((\rho^1\boxtimes\rho^2) (X,Y))( v\otimes u):=(\rho^1(X)v)\otimes u +v\otimes(\rho^2(Y)u).$$ 
If $\rho^1$ and $\rho^2$ are irreducible, so is $\rho^1\boxtimes \rho^2$. In fact, any irreducible complex representation of a nonsimple complex Lie algebra can be constructed by $\boxtimes$-products of irreducible representations of its simple factors.
Let $\rho^1\colon\mathfrak{g}_1\rightarrow \mathfrak{gl}(V)$ and $\rho^2\colon\mathfrak{g}_1\rightarrow \mathfrak{gl}(U)$ be two representations of $\mathfrak{g}_1$. We will denote by $\rho^1\otimes \rho^2\colon\mathfrak{g}_1\rightarrow \mathfrak{gl}(V\otimes U)$ the representation defined by $$((\rho^1\otimes\rho^2) (X))( v\otimes u):=(\rho^1(X)v)\otimes u +v\otimes(\rho^2(X)u).$$

A complex representation of a real Lie algebra $\mathfrak{g}$ is a homomorphism $\rho\colon\mathfrak{g}\rightarrow\mathfrak{gl}(V)$, where $V$ is a complex vector space. Two representations $\rho\colon\mathfrak{g}\rightarrow \mathfrak{gl}(V)$ and $\rho'\colon\mathfrak{g}\rightarrow \mathfrak{gl}(V')$ are said to be \emph{equivalent} if there exists a vector space isomorphism $\theta\colon V\rightarrow V'$ such that $\theta^{-1}\circ\rho'(X)\circ\theta=\rho(X)$ for all $X\in\mathfrak{g}$.
The following lemma allows us to use the powerful machinery of representation theory of complex semisimple Lie algebras on real semisimple Lie algebras, and in particular on Hermitian Lie algebras.
\begin{lma}
There is a one to one correspondence between (equivalence classes of) complex representations of a real semisimple Lie algebra $\mathfrak{g}$ and (equivalence classes of) complex representations of its complexification $\mathfrak{g}^{\mathbb{C}}=\mathfrak{g}\otimes_{\mathbb{R}}\mathbb{C}$. Further, irreducible representations of $\mathfrak{g}$ correspond to irreducible representations of $\mathfrak{g}^\mathbb{C}$.
\end{lma}
\begin{proof}
Given a representation $\rho\colon\mathfrak{g}\rightarrow~\mathfrak{gl}(V)$ we get a representation $\rho^\mathbb{C}\colon\mathfrak{g}^{\mathbb{C}}\rightarrow\mathfrak{gl}(V)$ by $\rho^{\mathbb{C}}(X\otimes z)v:=\rho(X)(zv)$.
From a representation $\rho\colon\mathfrak{g}^\mathbb{C}\rightarrow\mathfrak{gl}(V)$ we get a representation of $\mathfrak{g}$ by restriction. Since the adjoint action of $GL(V)$ and multiplication by complex numbers commute we get that the correspondence maps equivalence classes of representations to equivalence classes of representations.  
Given an irreducible representation $\rho\colon\mathfrak{g}\rightarrow\mathfrak{gl}(V)$ we get no invariant subspaces by acting with a larger algebra via $\rho^\mathbb{C}$. Thus $\rho$ irreducible implies $\rho^\mathbb{C}$ irreducible.
Starting with an irreducible representation $\rho\colon\mathfrak{g}^\mathbb{C}\rightarrow\mathfrak{gl}(V)$, assume that the restriction $\rho|\colon\mathfrak{g}\rightarrow\mathfrak{gl}(V)$ is not irreducible, i.e. there is an $\rho|$-invariant subspace $W\subset V$. But then $\rho(X\otimes z)W=\rho(X\otimes 1)zW=\rho|(X)zW=\rho|(X)W=W$ which contradicts that $\rho$ is irreducible.
 \end{proof}
We denote by $\rho_{m_1,...,m_n}$ the irreducible complex representation of a Hermitian Lie algebra $\mathfrak{g}$ that corresponds to the representation $\rho_{m_1,...,m_n}^\mathbb{C}$ of $\mathfrak{g}^\mathbb{C}$.

\subsection{Invariant forms}
Up to this point in the section, only homomorphisms into $\mathfrak{gl}(V)$ have been mentioned. We want to use representation theory to understand homomorphisms into $\mathfrak{su}(p,q)$. We do this by considering a homomorphism $\rho\colon\mathfrak{g}\rightarrow\mathfrak{su}(p,q)$ as a homomorphism $\rho\colon\mathfrak{g}\rightarrow\mathfrak{gl}(p+q,\mathbb{C})$ whose image is contained in a fixed subalgebra $\mathfrak{su}(p,q)\subset\mathfrak{gl}(p+q,\mathbb{C})$. An equivalent way of formulating this would be to say that $\rho$ is a representation on $V\simeq\mathbb{C}^{p+q}$ which leaves a nondegenerate Hermitian form $F$ of signature $(p,q)$ invariant. The following lemma and theorem tell us that, for the $\mathfrak{g}$:s that are of interest to us, any representation $V$ can be equipped with an invariant Hermitian form. Further, decompositions into irreducible representations are well behaved with respect to this additional structure. As the proofs of these results are rather long, and we suspect well known, we move them to the Appendix.

\begin{lma}\label{form}
Let $\rho\colon \mathfrak{g}\rightarrow \mathfrak{gl}(V)$ be an irreducible complex representation of a Hermitian Lie algebra $\mathfrak{g}$ containing only $\mathfrak{su}(p,q)$ and $\mathfrak{sp}(2m,\mathbb{R})$ as simple factors. Then there exists a nondegenerate Hermitian form $F$ on $V$ that is invariant under $\rho(\mathfrak{g})$. Furthermore, if $F_1$ and $F_2$ are two such forms then $F_1=cF_2$ for some $c\in\mathbb{R}$.
\end{lma}

\begin{thm}\label{43}
Let $U$ be a complex vector space equipped with a nondegenerate Hermitian form $F$. Further let $\rho\colon \mathfrak{g}\rightarrow \mathfrak{su}(U,F)$ be a homomorphism, with $\mathfrak{g}$ containing only $\mathfrak{su}(p,q)$ and $\mathfrak{sp}(2m,\mathbb{R})$ as simple factors. If $\rho$ is not an irreducible representation on $U$ there is a decomposition $U=\bigoplus_{i=1}^n U_i$, orthogonal with respect to $F$, such that $U_i$ is irreducible and $F|_{U_i}$ is nondegenerate for all $i$. This means that $\rho=(\rho^1,..., \rho^n)\colon\mathfrak{g}\rightarrow \bigoplus_i{\mathfrak{su}(r_i,s_i)}\subset \mathfrak{su}(U,F)$ where the inclusion is holomorphic and all $\rho^i$ are irreducible representations.
\end{thm}

\subsection{On equivalence}
Using representation theory to understand homomorphisms $\rho\colon\mathfrak{g}\rightarrow\mathfrak{su}(p,q)\subset\mathfrak{gl}(p+q,\mathbb{C})$ is not without problems.
Recall that two representations $\eta,\rho\colon\mathfrak{g}\rightarrow\mathfrak{gl}(p+q,\mathbb{C})$ are said to be equivalent if $\rho(\cdot)=\mbox{Ad}(g)\eta(\cdot)$ for some $g\in GL(p+q,\mathbb{C})$. We say that two homomorphisms $\eta,\rho\colon\mathfrak{g}\rightarrow\mathfrak{su}(p,q)$ are equivalent if $\rho(\cdot)=\mbox{Ad}(g)\eta(\cdot)$ for some $g\in SU(p,q)$. 
To differentiate between these two types of equivalence for homomorphisms $\eta,\rho\colon\mathfrak{g}\rightarrow\mathfrak{su}(p,q)\subset\mathfrak{gl}(p+q,\mathbb{C})$ we call the former \emph{equivalence as representations} and the latter \emph{equivalence as homomorphisms}. The following lemma describes the relation between these equivalences for an irreducible representation.
\begin{lma}\label{equivs}
Let $\rho\colon\mathfrak{g}\rightarrow\mathfrak{su}(p,q)=\mathfrak{su}(V,F)$ be an irreducible representation, where $\mathfrak{g}$ has simple factors of type $\mathfrak{sp}(2n)$ and $\mathfrak{su}(r,s)$. If $p\neq q$ the equivalence class of representations containing $\rho$ equals the equivalence class of homomorphisms containing $\rho$. If $p=q$ the equivalence class of representations containing $\rho$ contains two equivalence classes of homomorphisms. These are those of $\rho$ and $\theta\circ\rho$ where $\theta\colon\mathfrak{su}(p,p)\rightarrow\mathfrak{su}(p,p)$ is an antiholomorphic isomorphism.
\end{lma}
\begin{proof}
Fixing a suitable orthonormal basis for $V$, a matrix $X\in\mathfrak{su}(p,q)$ must satisfy
$$X^*I_{p,q}+I_{p,q}X=0$$
where $I_{p,q}$ is a diagonal matrix with the first $p$ entries equal to one and the last $q$ entries equal to minus one.
We then get the following matrix description for $\mathfrak{g}'=\mathfrak{su}(p,q)$
\begin{eqnarray*}
\mathfrak{g}'=\{ 
\left(\begin{array}{cc}
A&B\\
B^*&C
\end{array}\right) 
:A\in M_p(\mathbb{C}),B\in M_{p,q}(\mathbb{C}),C\in M_q(\mathbb{C}), \\A^*=-A, C^*=-C,\mbox{tr}(A)+\mbox{tr}(C)=0\}.
\end{eqnarray*}
We fix a Cartan decomposition $\mathfrak{g}'=\mathfrak{k}'+\mathfrak{p}'$ where 
\begin{eqnarray*}
\mathfrak{k}'=\{ 
\left(\begin{array}{cc}
A&0\\
0&C
\end{array}\right) \}\, ,
\mathfrak{p}'=\{ 
\left(\begin{array}{cc}
0&B\\
B^*&0
\end{array}\right) \}\, ,
\end{eqnarray*} 
The complex structure $Z_{\mathfrak{su}(p,q)}$ is then the diagonal matrix with the first $p$ entries equal to $\frac{iq}{p+q}$ and the last $q$ entries equal to $\frac{-iq}{p+q}$.

Now consider two homomorphisms $\rho,\rho'\colon\mathfrak{g}\rightarrow\mathfrak{su}(p,q)$ that are equivalent as irreducible representations, i.e. $\rho'(\cdot)=g\rho(\cdot)g^{-1}$ for some $g\in GL(p+q,\mathbb{C})$. With respect to the basis above $\rho$ and $\rho'$ must satisfy
\begin{eqnarray*}
\rho(X)^*I_{p,q}+I_{p,q}\rho(X)=0\\
\rho'(X)^*I_{p,q}+I_{p,q}\rho'(X)=0
\end{eqnarray*}
We can rewrite the second equation as
$(g^*)^{-1}\rho(X)^*g^*I_{p,q}+I_{p,q}g\rho(X)g^{-1}=0$ or equivalently
$\rho(X)^*g^*I_{p,q}g+g^*I_{p,q}g\rho(X)=0$.
Being irreducible $\rho$ only preserves the Hermitian forms $cF$ by Lemma \ref{form}. Thus $g^*I_{p,q}g=cI_{p,q}$ for some $c\in\mathbb{R}$. The center of $GL(p+q,\mathbb{C})$ acts trivially so we can assume that $g\in SL(p+q,\mathbb{C})$.
Taking determinants of $g^*I_{p,q}g=cI_{p,q}$ we get $c^{p+q}=1$. If $c=1$ then $g\in SU(p,q)$ and hence $\rho$ and $\rho'$ are equivalent as homomorphisms. 
Next we consider the possibility of $c=-1$. The signature of $g^*I_{p,q}g$ is $(p,q)$ for any $g\in GL(p+q,\mathbb{C})$ and the signature of $-I_{p,q}$ is $(q,p)$. Hence if $p\neq q$ no such $g$ exists, thus equivalence as representations implies equivalence as homomorphisms when $p\neq q$.
If $p=q$ we can find a $g$ realising $g^*I_{p,p}g=-I_{p,p}$, namely $g_0=
\left(\begin{array}{cc}
0 &I_p\\
-I_p&0
\end{array}\right)$.
Since $g_0\not\in SU(p,p)$ the equivalence induced by $g_0$ is as a representation but \emph{not} as a homomorphism. We thus have at least two equivalence classes of homomorphisms, $\rho$ and $g_0\rho(\cdot) g_0^{-1}$, in the equivalence class of representations of $\rho$. Suppose $g_1\rho(\cdot) g_1^{-1}$ is a third one. By the reasoning above $g_1$ must satisfy either 
$g_1^*I_{p,p}g_1=I_{p,p}$ or $g_1^*I_{p,p}g_1=-I_{p,p}$. In the first case $g_1\in SU(p,p)$ and hence $g_1\rho(\cdot) g_1^{-1}$ is equivalent as a homomorphism to $\rho$. In the second case we have $g_0^*(g_1^*I_{p,p}g_1)g_0=g_0^*(-I_{p,p})g_0=I_{p,p}$, i.e. $g_1g_0\in SU(p,p)$. This implies that $g_1\rho(\cdot) g_1^{-1}$ is equivalent as a homomorphism to $g_0\rho(\cdot) g_0^{-1}$. Thus there are exactly two equivalence classes of homomorphisms contained in the equivalence class of representations of an irreducible  representation $\rho\colon\mathfrak{g}\rightarrow\mathfrak{su}(p,p)$. We observe that $g_0Z_{\mathfrak{su}(p,p)}g_0^{-1}=-Z_{\mathfrak{su}(p,p)}$, i.e. $\theta:=\mbox{Ad}(g_0)\colon\mathfrak{su}(p,p)\rightarrow\mathfrak{su}(p,p)$ is an antiholomorphic isomorphism.
\end{proof}

The above lemma tells us that the equivalence class of representations of an irreducible representation $\rho$ contains at most two equivalence classes of homomorphisms, that of $\rho$ and that of $\theta\circ\rho$. By Lemma \ref{factor1} we have that $\rho$ is tight if and only if $\theta\circ\rho$ is tight. Thus tightness is well defined on equivalence classes of irreducible representations.

We should note here that tightness is \emph{not} well defined for equivalence classes of reducible representations. 
Using the matrix model for $\mathfrak{su}(p,q)$ from the proof above we illustrate this in the following example.
Let $\rho\colon\mathfrak{su}(1,1)\rightarrow\mathfrak{su}(2,2)$ be the homomorphism $$\rho(
\left(\begin{array}{cc}
a&b\\
\bar{b}&\bar{a}
\end{array}\right) 
)=
\left(\begin{array}{cccc}
a&0&b&0\\
0&a&0&b\\
\bar{b}&0&\bar{a}&0\\
0&\bar{b}&0&\bar{a}
\end{array}\right).$$
The representation $\rho$ is reducible so we can factor it as in Theorem \ref{43}, 
$\rho=\iota\circ(Id,Id)\colon\mathfrak{su}(1,1)\rightarrow\mathfrak{su}(1,1)\oplus\mathfrak{su}(1,1)\rightarrow\mathfrak{su}(2,2)$ where $\iota$ is the tight and holomorphic embedding 
 $$\iota(
\left(\begin{array}{cc}
a&b\\
\bar{b}&\bar{a}
\end{array}\right) ,\left(\begin{array}{cc}
c&d\\
\bar{d}&\bar{c}
\end{array}\right)
)=
\left(\begin{array}{cccc}
a&0&b&0\\
0&c&0&d\\
\bar{b}&0&\bar{a}&0\\
0&\bar{d}&0&\bar{c}
\end{array}\right)$$
and $(Id,Id)$ is the product of two identity homomorphisms.
Conjugating $\rho$ with 
$$g_0=\left(\begin{array}{cccc}
0&0&1&0\\
0&1&0&0\\
-1&0&0&0\\
0&0&0&1
\end{array}\right)\in GL(4,\mathbb{C})$$
we get the homomorphism
$\rho'\colon\mathfrak{su}(1,1)\rightarrow\mathfrak{su}(2,2)$, $$\rho'(
\left(\begin{array}{cc}
a&b\\
\bar{b}&\bar{a}
\end{array}\right) 
)=
\left(\begin{array}{cccc}
\bar{a}&0&-\bar{b}&0\\
0&a&0&b\\
-b&0&a&0\\
0&\bar{b}&0&\bar{a}
\end{array}\right).$$
We note here that $g_0\not\in SU(2,2)$, i.e. $\rho$ and $\rho'$ are equivalent as representations but not as homomorphisms.
We can factor $\rho'$ as 
$\rho'=\iota\circ(\theta,Id)\colon\mathfrak{su}(1,1)\rightarrow\mathfrak{su}(1,1)\oplus\mathfrak{su}(1,1)\rightarrow\mathfrak{su}(2,2)$
where $\theta$ is the antiholomorphic ismorphism
$$\theta(\left(\begin{array}{cc}
a&b\\ 
\bar{b}&\bar{a}
\end{array}\right))=
\left(\begin{array}{cc}
\bar{a}&-\bar{b}\\
-b&a
\end{array}\right).$$ 
The homomorphism $\iota$ is tight and holomorphic. The identity homomorphism is tight and positive, so the product of two identity homomorphisms is tight by Lemma \ref{factor2}. Thus $\iota\circ (Id,Id)=\rho$ is tight by Lemma \ref{factor3}. The antiholomorphic isomorphism $\theta$ is tight and negative. The product of the identity homomorphism and $\theta$ is thus nontight by Lemma \ref{factor2}. The composition $\iota\circ(\theta,Id)=\rho'$ is thus nontight by Lemma \ref{factor3}.

\subsection{Regular subalgebras}
In this subsection we give a brief description of regular subalgebras, both complex and Hermitian. For a thorough treatment of complex regular subalgebras see \cite[p.~142-151]{B15}, and for Hermitian regular subalgebras see \cite[p.~269-273]{B6}. 

Let $\mathfrak{g}_1^{\mathbb{C}}$ be a complex semisimple Lie algebra. A subalgebra $\mathfrak{g}_2^{\mathbb{C}}\subset\mathfrak{g}_1^{\mathbb{C}}$ is called a \emph{regular subalgebra} if there exists a Cartan subalgebra $\mathfrak{h}_1^{\mathbb{C}}\subset\mathfrak{g}_1^{\mathbb{C}}$ such that $\mathfrak{g}_2^{\mathbb{C}}$ is spanned by elements of $\mathfrak{h}_1^{\mathbb{C}}$ and root vectors relative to $\mathfrak{h}_1^{\mathbb{C}}$. If we restrict our attention to semisimple regular subalgebras these can be described in terms of a Cartan subalgebra and a subset of the root system. More precisely, let $\mathfrak{g}_1^{\mathbb{C}}=\mathfrak{h}_1^{\mathbb{C}}+\sum_{\alpha\in A}{\mathfrak{g}_\alpha}$ be the root space decomposition of $\mathfrak{g}_1^{\mathbb{C}}$ with respect to $\mathfrak{h}_1^\mathbb{C}$.
A subset $\Gamma\subset A$ satifying  
\begin{enumerate}
\item if $\alpha,\beta\in \Gamma$ then $\alpha -\beta \not\in A$,
\item $\Gamma$ is linearly independent in $(\mathfrak{h}_1^\mathbb{C})^*$,
\end{enumerate}
is called a \emph{Dynkin $\Pi$-system}. It defines a regular subalgebra as follows. Let $A(\Gamma):=\sum_{\gamma\in\Gamma}{\mathbb{Z}\gamma}\cap A$ and define $\mathfrak{g}^{\mathbb{C}}(\Gamma,\mathfrak{h}_1^\mathbb{C}):=\sum_{\gamma\in\Gamma}{\mathbb{C}H_\gamma}+\sum_{\gamma\in A(\Gamma)}{\mathfrak{g}_\gamma}$.

Let $\mathfrak{g}$ be a Hermitian Lie algebra with Cartan decomposition $\mathfrak{g}=\mathfrak{k}+\mathfrak{p}$. Choose a maximal abelian subalgebra $\mathfrak{h}\subset\mathfrak{k}$. The complexification $\mathfrak{h}^\mathbb{C}$ is then a Cartan subalgebra of $\mathfrak{g}^\mathbb{C}$. Since the complex structure $Z$ lies in the center of $\mathfrak{k}$ it is contained in any maximal abelian subalgebra $\mathfrak{h}$ and hence in $\mathfrak{h}^\mathbb{C}$. 
Since $ad(Z)^2(X)=\begin{cases}
 0,\mbox{ } X\in \mathfrak{k}^\mathbb{C}\\
 -X,\mbox{ } X\in\mathfrak{p}^\mathbb{C}
\end{cases}$
we know that each root space $\mathfrak{g}_\alpha$ is contained in either $\mathfrak{p}^\mathbb{C}$ or $\mathfrak{k}^\mathbb{C}$. We say that $\alpha$ is compact (respectively noncompact) if $\mathfrak{g}_\alpha\subset\mathfrak{k}^\mathbb{C}$ (respectively $\mathfrak{g}_\alpha\subset\mathfrak{p}^\mathbb{C}$).
We choose an ordering of the root system such that a noncompact root $\alpha$ is positive if $\alpha(Z)=i$. The root system will then have one noncompact root in each connected component of its Dynkin diagram. Let $A$ denote the root system of $\mathfrak{g}^\mathbb{C}$ with respect to an $\mathfrak{h}^\mathbb{C}$ chosen as above.
Let $\Gamma\subset A$ be a Dynkin $\Pi$-system further satisfying
\begin{enumerate}
\setcounter{enumi}{+2}
\item each connected component in the Dynkin diagram of $\Gamma$ contains exactly one noncompact root.
\end{enumerate}
$\Gamma$ defines a real subalgebra of $\mathfrak{g}$ by $\mathfrak{g}(\Gamma):=\mathfrak{g}^{\mathbb{C}}(\Gamma,\mathfrak{h}^\mathbb{C})\cap\mathfrak{g}$. The subalgebra $\mathfrak{g}(\Gamma)$ is called a \emph{regular subalgebra}. It is a Hermitian Lie algebra and can be equipped with a complex structure such that the inclusion homomorphism corresponds to a holomorphic map \cite[Proposition 3]{B6}. We drop the dependence of $\mathfrak{h}$ in the notation since different choices of $\mathfrak{h}\subset\mathfrak{k}$ will give isomorphic subalgebras whose inclusion homomorphisms are equivalent.

\section{Criteria for tightness}
In each Hermitian symmetric space $\mathcal{X}$ we can holomorphically and isometrically embed $\mathbb{D}^{r_\mathcal{X}}$, the product of $r_{\mathcal{X}}$ Poincar\'e discs, where $r_{\mathcal{X}}$ is the rank of $\mathcal{X}$. Composing this embedding with the holomorphic diagonal embedding $d\colon \mathbb{D}\rightarrow \mathbb{D}^{r_\mathcal{X}}$ we get a tight and holomorphic map from $\mathbb{D}$ into $\mathcal{X}$ known as a \emph{diagonal disc}. Since diagonal discs are holomorphic they depend on the complex structure of $\mathcal{X}$. We will abuse the notation a bit and refer to the corresponding Lie algebra homomorphisms as diagonal discs also. Diagonal discs are tight and play an important role in the following lemma from \cite[Lemma 8.1]{B8}.

\begin{lma}\label{diagd2}	
Let $\rho\colon\mathfrak{g}_1\rightarrow\mathfrak{g}_2$ be a homomorphism between Hermitian Lie algebras.
Further let $d_i\colon \mathfrak{su}(1,1)\rightarrow\mathfrak{g}_i$ be diagonal discs. Then $\rho$ corresponds to a tight and positive map if and only if 
\begin{equation}
\langle \rho \circ d_1 (Z_{\mathfrak{su}(1,1)}),Z_2\rangle_2 =\langle d_2 (Z_{\mathfrak{su}(1,1)}),Z_2\rangle_2.
\end{equation}
\end{lma}
Combining Lemma \ref{diagd2} with Lemma \ref{positive} we have the following corollary. We denote again the set of complex structures of a Hermitian symmetric space $\mathcal{X}_1$ by $\mathcal{J}_1$ and by $\mathcal{J}_1'$ a minimal subset of $\mathcal{J}_1$ fulfilling $\mathcal{J}_1'\cup -\mathcal{J}_1'=\mathcal{J}_1$. 

\begin{cor}\label{diagdcor}
Let $\rho\colon\mathfrak{g}_1\rightarrow\mathfrak{g}_2$ be a homomorphism between Hermitian Lie algebras.
Further let $d_1^J\colon \mathfrak{su}(1,1)\rightarrow(\mathfrak{g}_1,J)$ be diagonal discs for different complex structures $J\in\mathcal{J}_1'$ and $d_2\colon \mathfrak{su}(1,1)\rightarrow\mathfrak{g}_2$ a diagonal disc. Then $\rho$ corresponds to a tight map if and only if there is a $J\in\mathcal{J}_1'$ such that
\begin{equation*}
|\langle \rho \circ d_1^J (Z_{\mathfrak{su}(1,1)}),Z_2\rangle_2| =|\langle d_2 (Z_{\mathfrak{su}(1,1)}),Z_2\rangle_2|.
\end{equation*}
\end{cor}

\begin{thm}\label{nya}
Let $\rho\colon\mathfrak{g}\rightarrow \mathfrak{su}(r,s)$ be an irreducible representation, where $\mathfrak{g}$ is a Hermitian Lie algebra with simple factors of type $\mathfrak{su}(p,q)$ and $\mathfrak{sp}(2p)$. Further let $\iota\colon\mathfrak{g}_0\rightarrow\mathfrak{g}$ be a tight, injective and holomorphic homomorphism, where $\mathfrak{g}_0=\mathfrak{su}(1,1)$ or $\mathfrak{su}(1,1)\oplus \mathfrak{su}(1,1)$. Let $\rho\circ\iota=\sum{\rho^i}$ be a decomposition into irreducible representations. Each $\rho^i$ defines a homomorphism $\rho^i\colon\mathfrak{g}_0\rightarrow \mathfrak{su}(r_i,s_i)$ for some $(r_i,s_i)$. If there is one $\rho^i$ nontight then $\rho$ is nontight.
\end{thm}
\begin{proof}
Let $\mathcal{J}$ denote the set of complex structures of $\mathcal{X}$, the symmetric space associated to $\mathfrak{g}$, and $\mathcal{J}'$ a minimal subset such that $\mathcal{J}'\cup -\mathcal{J}'=\mathcal{J}$.
Assume for the moment that for each $J\in\mathcal{J}'$ there exists an embedding $\iota^J\colon\mathfrak{g}_0\rightarrow (\mathfrak{g},J)$ that is tight and holomorphic with respect to $J$. We assume further that $\iota^J$ is such that $\rho\circ\iota^J$ is equivalent to $\rho\circ\iota$ as a representation.
By Theorem \ref{43} we can factor each $\rho\circ\iota^J$ as \\
\centerline{\xymatrix{
\mathfrak{g}\ar[rr]^\rho &&\mathfrak{su}(r,s)\\
\mathfrak{g}_0\ar[u]^{\iota^J}\ar[rr]^{(\rho_J^1,...,\rho_J^m)} &&\oplus \mathfrak{su}(r_i,s_i)\ar[u]_{\iota_J'} 
}}
with $\iota_J'$ injective and holomorphic. Since all $\rho\circ\iota^J$ are equivalent as representations to $\rho\circ\iota$, $\rho^i_J$ is equivalent as a representation to $\rho^i$ for each $i,J$. By Lemma \ref{lma55} we know that $\iota_J'$ is strictly positive. If there is one $\rho^i$ nontight then $\rho^i_J$ is nontight for $J\in\mathcal{J}'$. This implies that $(\rho_J^1,...,\rho_J^m)$ is nontight for all $J\in\mathcal{J}'$ by Lemma \ref{factor2}.
Lemma \ref{factor3} then implies that $\iota_J'\circ(\rho_J^1,...,\rho_J^m)$ is nontight for alll $J\in\mathcal{J}'$.
since the diagram commutes this implies that $\rho\circ\iota^J$ is nontight for all $J\in\mathcal{J}'$. 
Lemma \ref{factor11} then implies that $\rho$ is nontight.
What remains to do is to construct $\iota^J$ with the properties described above. 

Let $\mathfrak{g}=\mathfrak{g}_1\oplus ...\oplus \mathfrak{g}_n$ be a decomposition into simple factors and $J=(J_1,...,J_n)$  the original complex structure for which $\iota=(\iota_1,...,\iota_n)$ is tight and holomorphic. We show that changing complex structure in one simple factor, say $\mathfrak{g}_1$, to $\hat{J}=(-J_1,J_2,...,J_n)$ we can construct $\hat{\iota}$ such that $\hat{\iota}\colon\mathfrak{g}_0\rightarrow (\mathfrak{g},\hat{J})$ is tight and holomorphic and such that $\rho\circ\hat{\iota}$ is equivalent to $\rho\circ\iota$ as a representation. We can then construct $\iota^J$ for any $J\in\mathcal{J}'$ by changing the complex structure in one factor at a time.

Let $\theta\colon\mathfrak{g}_0\rightarrow \mathfrak{g}_0$ be an antiholomorphic isomorphism. We construct our new tight and holomorphic embedding $\hat{\iota}:=(\iota_1\circ\theta,\iota_2,...,\iota_n)\colon \mathfrak{g}_0\rightarrow(\mathfrak{g},\hat{J})$. Define 
\begin{align*}
f\colon \mathfrak{g}_0\oplus \mathfrak{g}_0\rightarrow\mathfrak{g}, (X,Y)\mapsto (\iota_1(X),\iota_2(Y),...,\iota_n(Y)),\\ (Id,Id)\colon\mathfrak{g}_0\rightarrow\mathfrak{g}_0\oplus\mathfrak{g}_0, X\mapsto (X,X)\mbox{ and}\\ (\theta,Id)\colon\mathfrak{g}_0\rightarrow\mathfrak{g}_0\oplus\mathfrak{g}_0, X\mapsto (\theta(X),X). 
\end{align*}
We can factor our embeddings $\iota=f\circ (Id,Id)$, $\hat{\iota}=f\circ (\theta,Id)$. We need to show that the representations $\rho\circ\iota=\rho\circ f\circ (Id,Id)$ and  $\rho\circ\hat{\iota}=\rho\circ f\circ (\theta,Id)$ are equivalent. With the factorisations we see that it is sufficient to show that $\eta\circ (Id,Id)$ is equivalent to $\eta\circ (\theta,Id)$ for an arbitrary representation $\eta\colon \mathfrak{g}_0 \oplus \mathfrak{g}_0\rightarrow \mathfrak{gl}(V)$.

Let $\mathfrak{g}_0=\mathfrak{su}(1,1)\oplus \mathfrak{su}(1,1)$. An irreducible representation $(\rho_k,V^k)$ of $\mathfrak{su}(1,1)$ is the restriction of an irreducible representation of $\mathfrak{sl}(2,\mathbb{C})$. $V^k$ decomposes as $V^k=\oplus_{j=0}^k V_{k-2j}$ with $\rho_k(H)v_l=ilv_l$ for $v_l\in V_l$ and $H=i
\left(\begin{array}{cc}
1&0\\
0&-1
\end{array}\right)\in \mathfrak{su}(1,1)$.
An irreducible representation of $\mathfrak{g}_0^{\oplus 2}=\mathfrak{su}(1,1)^{\oplus 4}$ is the tensor product of four irreducible representations of $\mathfrak{su}(1,1)$, i.e. $\eta=\rho_{k_1}\boxtimes\rho_{k_2}\boxtimes\rho_{k_3}\boxtimes\rho_{k_4}\colon \mathfrak{su}(1,1)^{\oplus 4}\rightarrow \mathfrak{gl}(V^{k_1}\otimes V^{k_2}\otimes V^{k_3}\otimes V^{k_4})$. 
We show that any weight space of $\eta\circ (Id,Id)$ appears for $\eta\circ (\theta,Id)$ too. For each weight space $V_a\otimes V_b\otimes V_c\otimes V_d$ of $\eta$ there is also $V_{-a}\otimes V_{-b}\otimes V_c\otimes V_d$. We have
\begin{eqnarray*}
\eta\circ (Id,Id)(H,H) v_{a,b,c,d}=\eta(H,H,H,H) v_{a,b,c,d}\\
=\rho_{k_1}\boxtimes\rho_{k_2}\boxtimes\rho_{k_3}\boxtimes\rho_{k_4}(H,H,H,H)(v_a\otimes v_b\otimes v_c \otimes v_d)\\
=\rho_{k_1}(H)v_a \otimes v_b\otimes v_c \otimes v_d +...+v_a\otimes v_b \otimes v_c\otimes\rho_{k_4}(H)v_d\\
=i(a+b+c+d)v_{a,b,c,d}
\end{eqnarray*}
\thispagestyle{empty}
and 
\begin{eqnarray*}
\eta\circ (\theta,Id)(H,H) v_{-a,-b,c,d}=\eta (-H,-H,H,H)v_{-a,-b,c,d}\\
=\rho_{k_1}\boxtimes\rho_{k_2}\boxtimes\rho_{k_3}\boxtimes\rho_{k_4}(-H,-H,H,H)(v_{-a}\otimes v_{-b}\otimes v_c \otimes v_d)\\
=\rho_{k_1}(-H)v_{-a} \otimes v_{-b}\otimes v_c \otimes v_d +...+v_{-a}\otimes v_{-b} \otimes v_c\otimes\rho_{k_4}(H)v_d\\
=i(a+b+c+d)v_{-a,-b,c,d}
\end{eqnarray*}
Hence as representations of $\mathfrak{su}(1,1)\oplus \mathfrak{su}(1,1)$, $(\eta\circ(Id,Id), V^{k_1}\otimes ...\otimes V^{k_4})$ and $(\eta\circ(\theta,Id), V^{k_1}\otimes ...\otimes V^{k_4})$ decomposes into the same weight spaces. This implies that they are equivalent representations.
The case $\mathfrak{g}_0=\mathfrak{su}(1,1)$ is done in the same way.
 \end{proof}

We now have the tools we need to determine if a representation corresponds to a tight map.
We will also need the following elementary fact about holomorphic maps between Hermitian symmetric spaces. It says roughly that if a map between two irreducible Hermitian symmetric spaces is holomorphic in one direction, it is holomorphic in all directions.
\begin{lma}\label{nonholo}
Let $\rho\colon\mathfrak{g}_1\rightarrow\mathfrak{g}_2$ be a homomorphism between simple Hermitian Lie algebras.
If there is one nonzero vector $X\in\mathfrak{p}_1$ such that $\rho([Z_1,X])=[Z_2,\rho(X)]$ then this is true for all $X\in\mathfrak{p}_1$.
\end{lma}
\begin{proof}
 Consider $X'=[Y,X]$ for some $Y\in\mathfrak{k}_1$. Using that $Z_i$ is in the center of $\mathfrak{k}_i$ and the Jacobi identity we have
\begin{eqnarray*}
\rho([Z_1,X'])&=&\rho([Z_1,[Y,X]])=-\rho([X,[Z_1,Y]]+[Y,[X,Z_1]])\nonumber\\
&=&\rho([Y,[Z_1,X]])=[\rho(Y),\rho([Z_1,X])]= [\rho(Y),[Z_2,\rho(X)]]\nonumber\\
&=&-[\rho(X),[\rho(Y),Z_2]]-[Z_2,[\rho(X),\rho(Y)]]=[Z_2,\rho([Y,X])]\nonumber\\
&=&[Z_2,\rho(X')].\nonumber
\end{eqnarray*}
We thus have that $\rho$ is holomorphic in the $X'$-direction. The fact that $\mathfrak{p}_1$ is an irreducible representation of $\mathfrak{k}_1$ now proves the theorem.
 \end{proof}

\section{Some key low rank cases}
\subsection{Representations of $\mathfrak{su}(1,1)$}
The root system of $\mathfrak{sl}(2,\mathbb{C})$ is  $A=\{\alpha,-\alpha\}$ and the fundamental weight is $\omega=\frac{\alpha}{2}$. The irreducible complex representation $\rho_k$ of highest weight $k$ can be realised as the symmetric tensor product $V^{\odot k}$ of the standard representation $V\simeq \mathbb{C}^2$. 
Restricting these representations to the real form $\mathfrak{su}(1,1)$ we get the irreducible complex representations of $\mathfrak{su}(1,1)$. Let $\mathfrak{S}_k$ be the symmetric group. It acts on $V^{\otimes k}$ by permuting the factors, i.e. 
$$\sigma(v_1\otimes ... \otimes v_k):=v_{\sigma(1)}\otimes ... \otimes v_{\sigma(k)}.$$
For $v,w\in V$ we define 
$$v^lw^{k-l}:=\frac{1}{k!}\sum_{\sigma\in \mathfrak{S}_k}{\sigma(v\otimes ...\otimes v\otimes w\otimes ... \otimes w)}\in V^{\odot k}$$
where there are $l$ copies of $v$ and $k-l$ copies of $w$ in the product. Let $\{e_1,e_2\}$ be an orthonormal basis for $V$ with respect to $F$, the Hermitian form invariant under $\mathfrak{su}(1,1)$, $e_1$ being positive and $e_2$ negative. We can extend $F$ to a Hermitian form on $V^{\otimes k}$ in the natural way, 
$$F(v_1\otimes...\otimes v_k,w_1\otimes...\otimes w_k):=\prod_1^k{F(v_i,w_i)}.$$
$\mathfrak{su}(1,1)$ preserves this form. If $k=2l$ is even we have an orthogonal basis $$\{f_0,f_1,...,f_{k}\}:=\{e_1^k,e_1^{k-2}e_2^2,...,e_2^k,e_1^{k-1}e_2,e_1^{k-3}e_2^3,...,e_1e_2^{k-1}\}$$ for $V^{\odot k}$ with $\{f_0,f_1,...,f_{l}\}$ positive and $\{f_{l+1},f_{l+2},...,f_{k}\}$ negative. Hence $F$ is of signature $(l+1,l)$ and the representation defines a homomorphism into $\mathfrak{su}(l+1,l)$.
If $k=2l-1$ is odd we have an orthogonal basis $$\{h_1,...,h_{k+1}\}:=\{e_1^{2l-1},e_1^{2l-3}e_2^2,...,e_1e_2^{2l-2},e_1^{2l-2}e_2,e_1^{2l-4}e_2^3,...,e_2^{2l-1}\}$$ for $V^{\odot k}$ with $\{h_1,h_2,...,h_{l}\}$ positive and $\{h_{l+1},h_{l+2},...,h_{k}\}$ negative. Hence $F$ is of signature $(l,l)$ and the representation defines a homomorphism into $\mathfrak{su}(l,l)$.

\begin{thm}\label{suett}
Let $\rho_k\colon \mathfrak{su}(1,1)\rightarrow \mathfrak{su}(p,q)$ be an irreducible complex representation of highest weight $k$. Then $\rho_k$ is tight if and only if $k$ is odd.
\end{thm}
\begin{proof}
To see if these representations are tight we need to calculate the image of the complex structure. With respect to the basis $\{e_1,e_2\}$ we can write the complex structure as the matrix
$Z=\frac{i}{2}\left(\begin{array}{cc}
1&0\\
0&-1
\end{array}\right)$. 
The action on $V^{\odot k}$ is $\rho_k(Z) (e_1^{k-m} e_2^{m})=\frac{i}{2}(k-2m)(e_1^{k-m} e_2^{m})$. With respect to the basis $\{f_i\}$ we get the matrix 
$\rho_k(Z)=\frac{i}{2}\mbox{diag}(k,k-4,...,-k,k-2,k-6,...,2-k)=:\frac{i}{2}\left(\begin{array}{cc}
A_k&0\\
0&B_k
\end{array}\right)$ for $k$ even. The block form is with respect to the positive respectively the negative base vectors and will be used later. For $k$ odd we get
$\rho_k(Z)=\frac{i}{2}\mbox{diag}(k,k-4,...,2-k,k-2,k-6,...,-k) =:\frac{i}{2}\left(\begin{array}{cc}
C_k&0\\
0&D_k
\end{array}\right)$.
For $\mathfrak{su}(n,n)$ the complex structure is 
$Z_{\mathfrak{su}(n,n)}=\frac{i}{2}\left(\begin{array}{cc}
I_n&0\\
0&-I_n
\end{array}\right)$
and for $\mathfrak{su}(n+1,n)$ it is  
$Z_{\mathfrak{su}(n+1,n)}=\frac{i}{2n+1}\left(\begin{array}{cc}
nI_{n+1}&0\\
0&-(n+1)I_n
\end{array}\right)$.
A quick calculation shows that $|\langle \rho_{2l-1} \circ d_1 (Z), Z_{\mathfrak{su}(l,l)}\rangle|=|\langle d_2( Z), Z_{\mathfrak{su}(l,l)}\rangle|=\frac{l}{2}$ and $0=|\langle \rho_{2l}\circ d_1 (Z), Z_{\mathfrak{su}(l+1,l)}\rangle|\neq |\langle d_2 (Z), Z_{\mathfrak{su}(l+1,l)}\rangle|=\frac{l}{2}$. Hence by Corollary \ref{diagdcor} $\rho_k$ is tight if and only if $k$ is odd.
 \end{proof}

\subsection{Representations of $\mathfrak{su}(1,1)\oplus \mathfrak{su}(1,1)$}

An irreducible representation $(\rho,W)$ of $\mathfrak{su}(1,1)\oplus \mathfrak{su}(1,1)$ is the tensor product of two irreducible representations $(\rho_k,V)$ and $(\rho_l,U)$ of $\mathfrak{su}(1,1)$. 
If $V$ is equipped with a $\mathfrak{su}(1,1)$-invariant Hermitian form of signature $(a,b)$ and $U$ one of signature $(c,d)$ then $W=V\otimes U$ has a canonical $\mathfrak{su}(1,1)\oplus \mathfrak{su}(1,1)$-invariant Hermitian form of signature $(ac+bd,ad+bc)$, so $\rho$ defines a homomorphism into $\mathfrak{su}(ac+bd,ad+bc)$.
\begin{thm}\label{suett2}
Let $\rho=\rho_k\boxtimes\rho_l \colon \mathfrak{su}(1,1)\oplus \mathfrak{su}(1,1)\rightarrow \mathfrak{su}(r,s)=\mathfrak{su}(W,F)$ be an irreducible complex representation. Then $\rho$ is tight if and only if $k$ is odd and $l=0$ or vice versa.
\end{thm}
\begin{proof}
Let $\iota_1$ denote the diagonal embedding of $\mathfrak{su}(1,1)$ into $\mathfrak{su}(1,1) \oplus \mathfrak{su}(1,1)$, $X\mapsto (X,X)$, and $\iota_2$ the map $X\mapsto (X,-X^{t})$. Then $\iota_1$ is tight for $\mathfrak{su}(1,1)\oplus \mathfrak{su}(1,1)$ equipped with the complex structure $(Z,Z)$, and $\iota_2$ is tight for $\mathfrak{su}(1,1)\oplus \mathfrak{su}(1,1)$ equipped with the complex structure $(Z,-Z)$.

We first consider the case $k,l$ both odd or both even. Composing $\rho\circ \iota_i$ we get a representation of $\mathfrak{su}(1,1)$ that decomposes into irreducible representations of highest weights $k+l,k+l-2,...,|k-l|$, see \cite[p. 332]{B12}. All of these are of even highest weight hence $\rho$ is nontight by Theorems \ref{suett} and \ref{nya}.

Now consider the case of $l=2p-1$ odd and $k=2q$ even. We have that  $(r,s)=(p(q+1)+pq,pq+p(q+1))=(2pq+p,2pq+p)=(r,r)$.
Depending on the choice of complex structure on $\mathfrak{su}(1,1)\oplus \mathfrak{su}(1,1)$ we have two different diagonal discs $\iota_1$ and $\iota_2$.
We will now choose a natural basis for $W$ in terms of the bases $\{f_j\},\{h_i\}$ defined in the previous subsection. This basis is chosen such that $\rho\circ\iota_{1}(Z)$ and $\rho\circ\iota_{2}(Z)$ can be expressed by tensor products of $A_k,B_k,C_l,D_l$, the matrices defined in the proof of Theorem \ref{suett}, with appropriately sized identity matrices. 
Let 
\begin{eqnarray*}
\{e_{i+pj}\}&:=&\{f_j\otimes h_i | i=1,...,p\, , j=0,...,q\},\\
\{e_{i+pj-p}\}&:=&\{f_j\otimes h_i | i=p+1,...,2p\, ,j=q+1,...,2q\},\\
\{e_{i+pj+2pq}\}&:=&\{f_j\otimes h_i | i=p+1,...,2p\, ,j=0,...,q\},\\
\{e_{i+pj+p+pq}\}&:=&\{f_j\otimes h_i | i=1,...,p\, ,j=q+1,...,2q\}.
\end{eqnarray*}
We have that the first $2pq+p$ vectors are positive and the last $2pq+q$ are negative.
With this choice of basis we get the following matrix description 
\begin{eqnarray*}
\rho\circ \iota_1 (Z)=\frac{i}{2}\mbox{diag}(A_k\otimes I_{p}+I_{q+1}\otimes C_l,B_k\otimes I_{p}+I_{q}\otimes D_l,\\
A_k\otimes I_{p}+I_{q+1}\otimes D_l,B_k\otimes I_{p}+I_{q}\otimes C_l).
\end{eqnarray*}
Recall that for matrices $X,Y\in\mathfrak{su}(r,r)$ we have $\langle X,Y\rangle=4r\mbox{tr}(XY)$. We get
\begin{eqnarray*}
\frac{1}{4r}\langle \rho\circ \iota_1 (Z), Z_{\mathfrak{su}(r,r)}\rangle&=&
\frac{1}{4}(\mbox{tr}(A_k)p+\mbox{tr}(C_l)(q+1)+\mbox{tr}(B_k)p+\mbox{tr}(D_l)q-\\
&&\mbox{tr}(A_k)p-\mbox{tr}(D_l)(q+1)-\mbox{tr}(B_k)p-\mbox{tr}(C_l)q)\\
&=&\frac{1}{4}(\mbox{tr}(C_l)-\mbox{tr}(D_l))=\frac{1}{4}(p-(-p))=\frac{p}{2}.
\end{eqnarray*}
 
For the other choice of complex structure we have the matrix description
\begin{eqnarray*}
\rho\circ \iota_2 (Z)=\frac{i}{2}\mbox{diag}(A_k\otimes I_{p}+I_{q+1}\otimes -C_l,B_k\otimes I_{p}+I_{q}\otimes -D_l,\\
A_k\otimes I_{p}+I_{q+1}\otimes -D_l, B_k\otimes I_{p}+I_{q}\otimes -C_l)
\end{eqnarray*}
Thus
\begin{eqnarray*}
\frac{1}{4r}\langle \rho\circ \iota_2 (Z), Z_{\mathfrak{su}(r,r)}\rangle&=&\frac{1}{4}(\mbox{tr}(A_k)p-\mbox{tr}(C_l)(q+1)+\mbox{tr}(B_k)p-\mbox{tr}(D_l)q\\
&&-\mbox{tr}(A_k)p+\mbox{tr}(D_l)(q+1)-\mbox{tr}(B_k)p+\mbox{tr}(C_l)q)\\
&=&-\frac{1}{4}(\mbox{tr}(C_l)-\mbox{tr}(D_l))=-\frac{p}{2}.
\end{eqnarray*}

We have $$\frac{1}{4r}\langle  d_2 (Z), Z_{\mathfrak{su}(r,r)}\rangle=\frac{1}{4r}\langle  Z_{\mathfrak{su}(r,r)},Z_{\mathfrak{su}(r,r)}\rangle=\frac{p}{2}(2q+1).$$
From Corollary \ref{diagdcor} we have that $\rho$ is tight if and only if $\frac{p}{2}=\frac{p}{2}(2q+1)$ or equivalently $k=2q=0$.

 \end{proof}

\subsection{Representations of $\mathfrak{sp}(4,\mathbb{R})$}
For $\mathfrak{sp}(4,\mathbb{C})$ we have the root system $A=\{\pm\alpha_1,\pm\alpha_2,\pm(\alpha_1+\alpha_2),\pm(2\alpha_1+\alpha_2)\}$. We differ from the notation in \cite{B5} and \cite{B6} and let $\alpha_2$ denote the longer noncompact simple root. We have the fundamental weights $\omega_1=\frac{2\alpha_1+\alpha_2}{2}$ and $\omega_2=\alpha_1+\alpha_2$. From \cite[p.~451]{B7} we know that only the irreducible representation of highest weight $(1,0)$ corresponds to a holomorphic map.

\begin{thm}\label{sp4}
Let $\rho \colon \mathfrak{sp}(4,\mathbb{R})\rightarrow \mathfrak{su}(p,q)$ be an irreducible complex representation. If it is tight then it is (anti-) holomorphic.
\end{thm}
\begin{proof}

From \cite[p.~10]{B5} we know that there are two tight regular subalgebras of $\mathfrak{sp}(4,\mathbb{R})$. The first is $\mathfrak{g}(\alpha_1+\alpha_2)\simeq \mathfrak{su}(1,1)$ and the second is $\mathfrak{g}(\{\alpha_2, 2\alpha_1+\alpha_2\})=\mathfrak{g}(\alpha_2)\oplus\mathfrak{g}( 2\alpha_1+\alpha_2)\simeq \mathfrak{su}(1,1)\oplus \mathfrak{su}(1,1)$

We divide the representations $\rho$ into two types, the first type are the representations with highest weight $(0,l)$,
the second type is the remaining ones, with highest weight $(k,l)$ where $k\neq 0$. We exclude the case $(k,l)=(1,0)$ which is tight and (anti-) holomorphic.

To see that a representation of the first type is nontight we restrict it to the tightly embedded regular subalgebra $\mathfrak{g}(\alpha_1+\alpha_2)$. We then move over to the complexification of the algebras and the representation. We have $\rho^{\mathbb{C}}(H_{\alpha_1+\alpha_2})v_{0,l}=l(\alpha_1+\alpha_2)(H_{\alpha_1+\alpha_2})v_{0,l}=2lv_{0,l}$. So we know that $\rho^{\mathbb{C}}$ branches into at least one representation of $\mathfrak{g}(\alpha_1+\alpha_2)^{\mathbb{C}}$ with even highest weight. This means that $\rho$ is nontight by Theorems \ref{suett} and \ref{nya}.

To see that a representation of the second type is nontight we consider the restriction to the regular subalgebra $\mathfrak{g}(2\alpha_1+\alpha_2)$. Again we move over to the complexifications and we have $\rho^{\mathbb{C}}(H_{2\alpha_1+\alpha_2})v_{k,l}=(k\frac{2\alpha_1+\alpha_2}{2}+l(\alpha_1+\alpha_2))(H_{2\alpha_1+\alpha_2})v_{k,l}=(k+l) v_{k,l}$. We also have a weight $\omega_{k,l-1}$ which when restricted to the subalgebra becomes $k+l-1$. One of these numbers is even. If it is $k+l$ we know that it is even and nonzero, if it $k+l-1$ we know that it is nonzero since $(k,l)=(0,1)$ is of the first type and $(1,0)$ excluded. This means that when we branch $\rho$ to the subalgebra $\mathfrak{g}(\alpha_1)\oplus \mathfrak{g}(\alpha_1+2\alpha_2)$ there will appear even $j$:s in the decomposition $\rho|_{\mathfrak{g}(\alpha_1)\oplus\mathfrak{g}(\alpha_1 +2\alpha_2)}=\sum_{(i,j)\in I}{\rho_i\boxtimes \rho_j}\colon \mathfrak{su}(1,1)\oplus \mathfrak{su}(1,1)\rightarrow \mathfrak{su}(p,q)$. This means that $\rho$ is nontight by Theorems \ref{suett2} and \ref{nya}. 
 \end{proof}

\subsection{Representations of $\mathfrak{sp}(4,\mathbb{R})\oplus \mathfrak{su}(1,1)$}
\begin{thm}\label{sp4su1}

Let $\rho=\rho_{i,j}\boxtimes\rho_k\colon \mathfrak{sp}(4,\mathbb{R})\oplus \mathfrak{su}(1,1)\rightarrow \mathfrak{su}(p,q)$ be an irreducible complex representation.
Then $\rho$ is tight if and only if 
$$(i,j,k)=\begin{cases}
(1,0,0)\mbox{ or}\\
 (0,0,k)\mbox{ with $k$ odd.}
\end{cases}$$
In particular, this means that there are no irreducible tight representations $\rho\colon \mathfrak{sp}(4,\mathbb{R})\oplus \mathfrak{su}(1,1)\rightarrow \mathfrak{su}(p,q)$ that are both injective and nonholomorphic.
\end{thm}
\begin{proof}
By our previous analysis of representations of $\mathfrak{sp}(4,\mathbb{R})$ we know that if $(i,j)$ is not equal to $(1,0)$ or $(0,0)$ then either
\begin{enumerate}
\item there exists a diagonal disc $d\colon \mathfrak{su}(1,1)\rightarrow \mathfrak{sp}(4,\mathbb{R})$ such that $\rho_{i,j}\circ d=\sum{\rho_l}$ with some $l$ even and nonzero or
\item there exists a tight and holomorphic embedding $f\colon \mathfrak{su}(1,1)\oplus \mathfrak{su}(1,1)\rightarrow \mathfrak{sp}(4,\mathbb{R})$ such that $\rho_{i,j}\circ f=\sum{\rho_l\boxtimes \rho_m}$ with some $l$ even and nonzero.
\end{enumerate}

In the first case we consider the composition
\begin{equation*}
(\rho_{i,j}\boxtimes \rho_k) \circ (d\oplus Id ) \colon \mathfrak{su}(1,1)\oplus \mathfrak{su}(1,1)\rightarrow \mathfrak{sp}(4,\mathbb{R})\oplus \mathfrak{su}(1,1)\rightarrow \mathfrak{su}(p,q).
\end{equation*}
We get $(\rho_{i,j}\boxtimes \rho_k) \circ (d\oplus Id )=\sum{\rho_l}\boxtimes\rho_k$ with some $l$ even and nonzero. Hence some $\rho_l\boxtimes\rho_k$ is nontight by Theorem \ref{suett2}. This implies that $\rho$ is nontight by Theorem \ref{nya}. 

In the second case we begin by defining $$h\colon \mathfrak{su}(1,1)\oplus \mathfrak{su}(1,1)\rightarrow \mathfrak{sp}(4,\mathbb{R})\oplus \mathfrak{su}(1,1),\, \,(X,Y)\mapsto (f(X,Y),Y).$$ This a tight and holomorphic embedding.
We have 
\begin{eqnarray*}
\rho\circ h(X,Y)&=&\rho_{i,j}\boxtimes \rho_k(f(X,Y),Y)=(\sum_{l,m}{\rho_l\boxtimes \rho_m})\boxtimes \rho_k(X,Y,Y)\\
&=&\sum_{l,m}{\rho_l\boxtimes (\rho_m\otimes \rho_k)}(X,Y).
\end{eqnarray*}
However, $\rho_m\otimes\rho_k$ is not irreducible and can thus be written as $\sum_{n\in I_{m,k}}{\rho_n}$. We thus get 
\begin{eqnarray*}
\rho\circ h=\sum_{l,m}{\rho_l\boxtimes (\rho_m\otimes \rho_k)} 
=\sum_{l,m}\sum_{n\in I_{m,k}}{\rho_l\boxtimes \rho_n}.
\end{eqnarray*}
 Since some $l$ is even and nonzero, we know by Theorems \ref{suett2} and \ref{nya} that $\rho$ is nontight.
That the map is tight for the remaining cases $(i,j,k)=(1,0,0)$ and $(0,0,odd)$ follows from the classification in \cite{B5} and \cite{B8}.
 \end{proof} 

\subsection{Representations of $\mathfrak{su}(2,1)$}

Let $\mathfrak{su}(2,1)^{\mathbb{C}}=\mathfrak{sl}(3,\mathbb{C})=\mathfrak{h}+\sum_{\alpha\in A}{\mathfrak{g}_{\alpha}}$ be a root space decomposition of $\mathfrak{sl}(3,\mathbb{C})$. Here $A=\{\pm\alpha_1,\pm\alpha_2,\pm(\alpha_1+\alpha_2)\}$ where $\alpha_1$ is chosen as the noncompact simple root. The fundamental weights are $\omega_1=\frac{2\alpha_1+\alpha_2}{3}$ and $\omega_2=\frac{\alpha_1+2\alpha_2}{3}$.
The Weyl group of $\mathfrak{sl}(3,\mathbb{C})$ consists of six elements. These send a weight $\omega_{k,l}$ to $\omega_{k,l}$, $\omega_{-k-l,l}$,  $\omega_{l,-k-l}$, $\omega_{-k,k+l}$, $\omega_{k+l,-l}$ and $\omega_{-l,-k}$ respectively.

We recall from \cite[p.~447]{B7} that the only irreducible representations of $\mathfrak{su}(2,1)$ that are (anti-) holomorphic have highest weight $(1,0)$ and $(0,1)$. 
From \cite[p.~8]{B5} we know that there is a tight regular subalgebra $\mathfrak{g}(\alpha_1)\cong \mathfrak{su}(1,1)$. 
\begin{thm}\label{su21}
Let $\rho\colon \mathfrak{su}(2,1)\rightarrow \mathfrak{su}(p,q)$ be an irreducible complex representation. If $\rho$ is tight then it is (anti-) holomorphic.
\end{thm}
\begin{proof}
Let $\rho_{k,l}$ be an irreducible representation of $\mathfrak{su}(2,1)$ with $k+l\geq 2$. We restrict this representation to $\mathfrak{g}(\alpha_1)$. We now look at the complexifications $\rho_{k,l}^{\mathbb{C}}$ and $\rho_{k,l}^{\mathbb{C}}|_{\mathfrak{g}^{\mathbb{C}}(\alpha_1)}$ to see which weights appear. We divide our analysis of the branching into two cases. 

In the first we assume that either $k=0$ or $l=0$.
Along the line between the weights $\omega_{k,l}$ and $\omega_{-l,-k}$ are the weights $\omega_{k,l}-n(\alpha_1+\alpha_2)=\omega_{k,l}-n(\omega_1+\omega_2)=\omega_{k-n,l-n}$ for $n=1,...,k+l$. Since $k+l\geq 2$ we have in particular that the weights $\omega_{k-1,l-1}$ and $\omega_{k-2,l-2}$ appear in the representation $\rho_{k,l}^{\mathbb{C}}$.
Assume $l=0$, we have that $\omega_{k,l}(H_{\alpha_1})=k$ and $\omega_{k-1,l-1}(H_{\alpha_1})=k-1$. One of these integers must be even and nonzero since $k>1$. If instead $k=0$ we have $\omega_{k-2,l-2}(H_{\alpha_1})=-2$. This means that $\rho|_{\mathfrak{g}(\alpha_1)}$ contains irreducible representations of even highest weight. Hence $\rho$ is nontight by Theorems \ref{suett} and \ref{nya}.

We now consider the case $k\neq 0$, $l\neq 0$.
Along the line between the weights $\omega_{k,l}$ and $\omega_{k+l,-l}$ are the weights $\omega_{k,l}-n\alpha_2=\omega_{k,l}-n(2\omega_2-\omega_1)=\omega_{k+n,l-2n}$ for $n=1,...,l$. Since $l\geq 1$ we have in particular that the weight $\omega_{k+1,l-2}$  appear in the representation $\rho_{k,l}^{\mathbb{C}}$. We have that $\omega_{k,l}(H_{\alpha_1})=k$ and $\omega_{k+1,l-2}(H_{\alpha_1})=k+1$ and that one of these integers must be even and nonzero since $k\geq 1$. Again this implies that $\rho$ is nontight by Theorems \ref{suett} and \ref{nya}.
 \end{proof}
Although it was not needed here the curious reader can find an explicit construction of the irreducible representations of $\mathfrak{su}(2,1)$ in the appendix as part of the proof of Lemma \ref{form}.

\section{Proof of Main Theorem}
Before proving Theorem \ref{main} we recall
the following two theorems that are consequences of the classification in \cite{B5}.
\begin{thm}\label{holo}
Let $\mathfrak{g}$ be a simple Hermitian Lie algebra, $\mathfrak{g}\neq\mathfrak{su}(1,1)$. Then there exists a tight, injective  and holomorphic homomorphism from either $\mathfrak{sp}(4,\mathbb{R}),\mathfrak{sp}(4,\mathbb{R})\oplus \mathfrak{su}(1,1)$ or $\mathfrak{su}(2,1)$ into $\mathfrak{g}$.
\end{thm}
\begin{proof}
We start with the case $\mathfrak{g}=\mathfrak{sp}(2m,\mathbb{R})$. We can assume that $m\geq 2$, since if $m=1$ we have $\mathfrak{sp}(2,\mathbb{R})\simeq \mathfrak{su}(1,1)$. If $m=2$ the identity homomorphism will suffice.
For $m=3$ we see in \cite[p.~10]{B5} that $\mathfrak{sp}(6,\mathbb{R})$ contains a tight regular subalgebra $\mathfrak{sp}(4,\mathbb{R})\oplus \mathfrak{su}(1,1)$.
For $m>3$ $\mathfrak{sp}(2m,\mathbb{R})$ contains a tight regular subalgebra $\mathfrak{sp}(2(m-2),\mathbb{R})\oplus \mathfrak{sp}(4,\mathbb{R})$ \cite[p. ~10]{B5}. Denote the inclusion of this subalgebra by $\iota$. If $m=2k$ is even we have by induction on $k$ a tight and holomorphic embedding $\rho_{k-1}\colon\mathfrak{sp}(4,\mathbb{R})\rightarrow\mathfrak{sp}(4(k-1),\mathbb{R})$. Define $\rho_{k}\colon\mathfrak{sp}(4,\mathbb{R})\rightarrow\mathfrak{sp}(4(k-1),\mathbb{R})\oplus\mathfrak{sp}(4,\mathbb{R})\rightarrow\mathfrak{sp}(4k,\mathbb{R})$ by $\rho_k(X)=\iota(\rho_{k-1}(X),X)$.
If $m=2k+1$ is odd we have by induction on $k$ a tight and holomorphic embedding $\rho_{k-1}\colon\mathfrak{sp}(4,\mathbb{R})\oplus\mathfrak{su}(1,1)\rightarrow\mathfrak{sp}(4(k-1)+2,\mathbb{R})$. Define $\rho_{k}\colon\mathfrak{sp}(4,\mathbb{R})\oplus\mathfrak{su}(1,1)\rightarrow\mathfrak{sp}(4(k-1)+2,\mathbb{R})\oplus\mathfrak{sp}(4,\mathbb{R})\rightarrow\mathfrak{sp}(4k+2,\mathbb{R})$ by $\rho_k(X,Y)=\iota(\rho_{k-1}(X,Y),X)$.
That the composition of these maps is tight follows from Lemmas \ref{lma55} and \ref{factor1}.

If $\mathfrak{g}\neq\mathfrak{sp}(2m,\mathbb{R})$ and of rank at least two we can see in \cite[p.~8-12]{B5} that in each case there is a tight regular subalgebra of $\mathfrak{g}$ isomorphic to $\mathfrak{su}(n,n)$ with $n\geq 2$. Compose this inclusion with the tight and holomorphic map $\mathfrak{sp}(2n,\mathbb{R})\rightarrow\mathfrak{su}(n,n)$ \cite[p.~14]{B5}. Then further compose with the tight and holomorphic inclusion of $\mathfrak{sp}(4,\mathbb{R})$ or $\mathfrak{sp}(4,\mathbb{R})\oplus\mathfrak{su}(1,1)$ into $\mathfrak{sp}(2n,\mathbb{R})$ as described above. The composition of them all then yields the desired result.

The remaining case is when $\mathfrak{g}$ is of rank one and not isomorphic to $\mathfrak{su}(1,1)$. This means that $\mathfrak{g}\simeq\mathfrak{su}(m,1)$ and from \cite[p.~9]{B5} we know that it contains $\mathfrak{su}(2,1)$ as a tight regular subalgebra.   
\end{proof}

\begin{thm}[{\cite[p.~14-15]{B5}}]\label{holo2}
Let $\mathfrak{g}$ be a classical simple Hermitian Lie algebra of tube type. Then there exists a tight and holomorphic homomorphism $\rho\colon\mathfrak{g}\rightarrow\mathfrak{su}(n,n)$ for some $n$.
\end{thm}
Every Hermitian Lie algebra $\mathfrak{g}$ contains maximal (with respect to inclusion) subalgebras $\mathfrak{g}^T$ of tube type which are holomorphically embedded and of the same rank as $\mathfrak{g}$. The inclusion $\mathfrak{g}^T\subset\mathfrak{g}$ is tight and $\mathfrak{g}^T$ is unique up to equivalence. Tube type subalgebras are important as we see in the following structure theorem from \cite[Theorem 9]{B8}.
\begin{thm}\label{factor0}
If $\mathfrak{g}'$ is of tube type and $\rho\colon\mathfrak{g}' \rightarrow\mathfrak{g}$ is a tight homomorphism then there is a maximal tube type subalgebra $\mathfrak{g}^T\subset\mathfrak{g}$ such that  $\rho(\mathfrak{g}')\subset\mathfrak{g}^T$.
\end{thm}

We now have the tools we need to prove Theorem \ref{main}.  We will do this through a series of lemmas. 

\begin{lma}\label{ett}
Let $\rho\colon \mathfrak{g}\rightarrow \mathfrak{su}(p,q)$ be a homomorphism between simple Hermitian Lie algebras, where $\mathfrak{g}\neq \mathfrak{su}(1,1)$. If $\rho$ is tight then it is (anti-) holomorphic. 
\end{lma}

\begin{proof}
Assume that $\rho$ is tight and nonholomorphic.
By Theorem \ref{holo} there exists a tight and holomorphic embedding $\iota\colon\mathfrak{g}'\rightarrow\mathfrak{g}$, where $\mathfrak{g}'=\mathfrak{su}(2,1)$, $\mathfrak{sp}(4,\mathbb{R})$ or $\mathfrak{sp}(4,\mathbb{R})\oplus \mathfrak{su}(1,1)$.
We can decompose $\rho\circ\iota$ into irreducible representations and factor it using Theorem \ref{43} to get the following commutative diagram.\\
\centerline{\xymatrix{
\mathfrak{g}\ar[rr]^\rho &&\mathfrak{su}(r,s)\\
\mathfrak{g}'\ar[u]^{\iota}\ar[rr]^{(\rho^1,...,\rho^n)} &&\oplus \mathfrak{su}(r_i,s_i)\ar[u]_{\iota'} 
}}
Here $\iota'$ is a holomorphic embedding. If $\rho$ is tight and nonholomorphic, we have by Lemma \ref{factor12} that $\rho\circ\iota=\iota'\circ (\rho^1,...,\rho^n)$ is tight, and by Lemma \ref{nonholo} it is nonholomorphic when restricted to any simple factor of $\mathfrak{g}'$. By Lemma \ref{factor2} and \ref{factor3} tightness of $\iota'\circ(\rho^1,...,\rho^n)$ implies that either all $\rho^i$ are tight and positive or all $\rho^i$ are tight and negative.

If $\mathfrak{g}'=\mathfrak{su}(2,1)$ or $\mathfrak{sp}(4,\mathbb{R})$, this implies that either all $\rho^i$ are tight and holomorphic or all $\rho^i$ are tight and antiholomorphic by Theorem \ref{su21} or \ref{sp4}. This contradicts that $\iota'\circ(\rho^1,...,\rho^n)$ is nonholomorphic.
 
If $\mathfrak{g}'=\mathfrak{sp}(4,\mathbb{R})\oplus\mathfrak{su}(1,1)$, tightness implies that all $\rho^i$ are of the form
$\rho_{0,0}\boxtimes\rho_{odd}$ and $\rho_{1,0}\boxtimes\rho_0$. For both of these types of $\rho^i$ we have that $r_i=s_i$. By Lemma \ref{equivs} the equivalence class of representations of $\rho_{1,0}\boxtimes\rho_0$ contains both holomorphic and antiholomorphic homomorphisms. If all $\rho^i$ are positive, respectively negative, this implies that all $\rho_{1,0}\boxtimes\rho_0$  are holomorphic, respectively antiholomorphic. As the $\rho_{0,0}\boxtimes\rho_{odd}$ representations are trivial in the $\mathfrak{sp}(4,\mathbb{R})$-factor they are always always both holomorphic and antiholomorphic in the $\mathfrak{sp}(4,\mathbb{R})$-factor. Thus $\iota'\circ(\rho^1,...,\rho^n)=\rho\circ\iota$ is either holomorphic or antiholomorphic in the $\mathfrak{sp}(4,\mathbb{R})$-factor which is a contradiction.
 \end{proof}

\begin{lma}\label{tva}
Let $\rho\colon \mathfrak{g}_1\rightarrow\mathfrak{g}_2$ be a homomorphism between simple Lie algebras. Assume that $\mathfrak{g}_2$ is classical and that $\mathfrak{g}_1\neq \mathfrak{su}(1,1)$ and of tube type. If $\rho$ is tight then it is (anti-) holomorphic.
\end{lma}
\begin{proof}
Assume that $\rho$ is tight and nonholomorphic. By Theorem \ref{factor0} there is a subalgebra $\mathfrak{g}_2^T\subset\mathfrak{g}_2$ of tube type containing the image of $\mathfrak{g}_1$. We can thus factor $\rho$ as $\iota\circ\rho'\colon\mathfrak{g}_1\rightarrow\mathfrak{g}_2^T\rightarrow\mathfrak{g}_2$ where $\iota$ is the inclusion homomorphism $\iota\colon\mathfrak{g}_2^T\rightarrow \mathfrak{g}_2$ and $\rho'$ is $\rho$ with restricted codomain. Since $\mathfrak{g}_2$ is simple, $\mathfrak{g}_2^T$ is too. That $\iota$ and $\iota\circ\rho'$ are tight implies that $\rho'$ is tight by Lemma \ref{factor12}. Since $\iota$ is holomorphic and $\iota\circ\rho'$ is nonholomorphic, $\rho'$ must be nonholomorphic. By Theorem \ref{holo2} we can find a tight and holomorphic map $\eta\colon\mathfrak{g}_2^T\rightarrow\mathfrak{su}(n,n)$ for some $n$. The composition $\eta\circ\rho'$ will then be nonholomorphic and tight by Lemma \ref{factor12}. But this contradicts Lemma \ref{ett}, hence $\rho$ can not be tight and nonholomorphic.
 \end{proof}
\begin{lma}\label{bla}
There are no tight homomorphisms $\rho\colon \mathfrak{su}(n,1)\rightarrow \mathfrak{so}^*(2p)$ for $n\geq 2$, $p\geq 4$.
\end{lma}
\begin{proof}
Consider the composition $\rho\circ\iota\colon \mathfrak{su}(2,1)\rightarrow \mathfrak{su}(n,1)\rightarrow \mathfrak{so}^*(2p)$ where $\iota$ is the canonical tight inclusion. Then $\rho$ is tight if and only if $\rho\circ\iota$ is tight by Lemma \ref{factor12}. This reduces our search of tight homomorphisms to the case $n=2$.

If $p$ is even we consider the composition $\iota\circ\rho\colon \mathfrak{su}(2,1)\rightarrow \mathfrak{so}^*(2p)\rightarrow \mathfrak{so}^*(2(p+1))$, where $\iota$ is the canonical tight inclusion. Then $\rho$ is tight if and only if $\iota\circ\rho$ is tight. This reduces our search of tight homomorphisms to the case $p\geq 5$ and odd.

Consider the restriction of $\rho$ to a tightly embedded regular subalgebra isomorphic to $\mathfrak{su}(1,1)$, as in the proof of Theorem \ref{su21}. Assume that $\rho$ is tight. Then the image of $\mathfrak{su}(1,1)$ is contained in a maximal subalgebra of $\mathfrak{so}^*(2p)$ that is of tube type by Theorem \ref{factor0}. From \cite{B5} we know that this algebra is $\mathfrak{so}^*(2(p-1))$. This algebra can be tightly and holomorphically embedded in $\mathfrak{su}(p-1,p-1)$. We can extend this to a homomorphism $i\colon \mathfrak{so}^*(2p)\rightarrow \mathfrak{su}(p,p)$, though this extension is not tight. We get the following commutative diagram of homomorphisms.\\
\centerline{\xymatrix{
\mathfrak{su}(2,1)\ar[r]^\rho &\mathfrak{so}^*(2p)\ar[r]^i &\mathfrak{su}(p,p)\\
\mathfrak{su}(1,1)\ar[u]^{\iota_1}\ar[r]^{\rho|} &\mathfrak{so}^*(2(p-1))\ar[r]^{i|}\ar[u]_{\iota_2} &\mathfrak{su}(p-1,p-1)\ar[u]^{\iota_3}
}}
We have by Lemma \ref{factor12} and commutativity of the diagram the following equivalences:
$$\rho \mbox{ tight } \Leftrightarrow\rho\circ\iota_1 \mbox{ tight }\Leftrightarrow\iota_2\circ\rho| \mbox{ tight } \Leftrightarrow \rho| \mbox{ tight } \Leftrightarrow i|\circ\rho| \mbox{ tight}$$
As a representation, $i$ is just the the standard representation. So $i|\circ\rho|$ decomposes into the same irreducible parts as $\rho|$. Since it is tight, we have by Theorem \ref{suett} the decomposition into irreducible representations
$$i|\circ\rho|=\sum_{k\mbox{ odd}}{n_k\rho_k}+n_0\rho_0.$$ 
Here the $n_k$ denote the multiplicities of the representations. Composing with $\iota_3$ adds two trivial representations
$$\iota_3\circ i|\circ\rho|=\sum_{k\mbox{ odd}}{n_k\rho_k}+(n_0+2)\rho_0.$$ 
From the proof of Theorem \ref{su21} we know that the only representations of $\mathfrak{su}(2,1)$ that branches into odd- and zero-highest weight representations when restricted to $\mathfrak{su}(1,1)$ are $\rho_{1,0}$, $\rho_{0,1}$ and $\rho_{0,0}$. The first two both branch into $\rho_1+\rho_0$ and the last one to $\rho_0$ when restricted to $\mathfrak{su}(1,1)$.
Thus $i\circ\rho$ is of the form $i\circ\rho=k\rho_{1,0}+(n-k)\rho_{0,1}+l\rho_{0,0}$.
We get $i\circ\rho\circ\iota_1=n\rho_1+(n+l)\rho_0$.
Since the diagram commutes we have
$$n\rho_1+(n+l)\rho_0=\sum_{k\mbox{ odd}}{n_k\rho_k}+(n_0+2)\rho_0.$$ 
Thus $n_1=n$, $n_0+2=n+l$ and $n_k=0$ for $k\neq 1,0$. Since $i|\circ\rho|$ is tight we have that $n_1=n=p-1$, for dimensional reasons we have $3n+l=2p$. These equalities together give us $p-3+l=0$. But $p\geq 5$ and $l$ is a nonnegative integer. We thus get a contradiction. Hence $\rho$ can not be tight.
 \end{proof}

\begin{proof}[Proof of Theorem \ref{main}]
Lemma \ref{tva} covers the case when $\mathfrak{g}_1$ is of tube type. It remains to show that if $\rho\colon\mathfrak{g}_1\rightarrow\mathfrak{g}_2$ is a tight homomorphism where $\mathfrak{g}_1$ is not of tube type, then $\rho$ must be (anti-) holomorphic.

Consider first the case $\mbox{rank}(\mathfrak{g}_1)\geq 2$.
We know that $\mathfrak{g}_1$ contains a maximal tube type subalgebra $\mathfrak{g}_1^T$ with $\mbox{rank}(\mathfrak{g}_1^T)=\mbox{rank}(\mathfrak{g}_1)$. The inclusion homomorphism $\iota\colon\mathfrak{g}_1^T\rightarrow\mathfrak{g}_1$ is tight and holomorphic. If $\rho$ is tight and nonholomorphic the composition $\rho\circ\iota$ will also be tight and nonholomorphic by Lemma \ref{factor12}. But this contradicts Lemma \ref{tva}, hence $\rho$ can not be tight and nonholomorphic.

Next we treat the case $\mbox{rank}(\mathfrak{g}_1)=1$ and $\mathfrak{g}_2$ of tube type.
Since $\mathfrak{g}_2$ is of tube type there is a tight and holomorphic homomorphism $\iota\colon\mathfrak{g}_2\rightarrow\mathfrak{su}(n,n)$ by Theorem \ref{holo2}. If $\rho\colon\mathfrak{g}_1\rightarrow\mathfrak{g}_2$ is tight and nonholomorphic then $\iota\circ\rho$ will also be tight and nonholomorphic by Lemma \ref{factor12}. This contradicts Lemma \ref{ett}, hence $\rho$ can not be tight and nonholomorphic. 

Finally, consider the case $\mbox{rank}(\mathfrak{g}_1)=1$ and $\mathfrak{g}_2$ not of tube type.
That the rank of $\mathfrak{g}_1$ is one implies that $\mathfrak{g}_1$ is isomorphic to $\mathfrak{su}(m,1)$ for some $m\geq 2$. There are two possibilities for $\mathfrak{g}_2$, either $\mathfrak{g}_2=\mathfrak{su}(p,q)$ with $p>q$ or $\mathfrak{g}_2=\mathfrak{so}^{*}(2p)$ with $p$ odd.
The first case is covered by Lemma \ref{ett}. 
Since $\mathfrak{so}^*(6)\simeq \mathfrak{su}(3,1)$ we can assume $p\geq 5$ in the second case. It is thus covered by Lemma \ref{bla}.
 \end{proof}

\begin{proof}[Proof of Theorem \ref{bonus}]
Assume that $\rho\colon\mathfrak{g}\rightarrow \mathfrak{e}_{6 (-14)}$ is tight and nonholomorphic. Let $\mathfrak{g}^T$ be a maximal tube type subalgebra of $\mathfrak{g}$. Denote by $\iota\colon\mathfrak{g}^T\rightarrow\mathfrak{g}$ the tight and holomorphic inclusion homomorphism. Since $\rho$ is tight and nonholomorphic, $\rho\circ\iota$ is tight and nonholomorphic by Lemma \ref{factor12}. By Theorem \ref{factor0} $\rho\circ\iota(\mathfrak{g}^T)\subset \mathfrak{e}_{6(-14)}^T$, where $\mathfrak{e}_{6(-14)}^T$ is a maximal tube type subalgebra of $\mathfrak{e}_{6(-14)}$. Since $\mathfrak{e}_{6(-14)}$ is not of tube type, $\mathfrak{e}_{6(-14)}^T$ is a proper subalgebra and hence must be classical. Indeed, $\mathfrak{e}_{6(-14)}^T=\mathfrak{so}(2,8)$ checking root multiplicities. 
We can thus consider $\rho\circ\iota$ as a homomorphism $\rho\circ\iota\colon\mathfrak{g}^T\rightarrow\mathfrak{e}^T_{6(-14)}$. This homomorphism is tight and nonholomorphic by Lemma \ref{factor12}. But this contradicts Lemma \ref{tva}, hence $\rho$ can not be tight and nonholomorphic.
 \end{proof}

\section{Appendix}
\begin{proof}[Proof of Lemma \ref{form}]
We begin with the case $\mathfrak{g}=\mathfrak{su}(p,q)$. The algebra $\mathfrak{su}(p,q)$ is defined as the Lie algebra of trace-free linear endomorphisms of a complex vector space $V\simeq\mathbb{C}^{p+q}$ preserving a nondegenerate Hermitian form $F$ of signature $(p,q)$. Let $e_1,...,e_{p+q}$ denote an orthonormal basis for $(V,F)$ and let $n+1=p+q$. The complexification of $\mathfrak{su}(p,q)$ is then $\mathfrak{sl}(n+1,\mathbb{C})$. The action of $\mathfrak{su}(p,q)$ on $V$ extends naturally to an action of $\mathfrak{sl}(n+1,\mathbb{C})$ on $V$. We call $V$ the \emph{standard representation} of $\mathfrak{sl}(n+1,\mathbb{C})$. Let $\omega_1,...,\omega_n$ denote the set of fundamental weights for $\mathfrak{sl}(n+1,\mathbb{C})$. The irreducible representations of highest weights $\omega_1,...,\omega_n$ can be realised by $V,V\wedge V,...,V^{\wedge n}$, \cite[p.~223]{B16}. The restrictions of these representations to $\mathfrak{su}(p,q)$ then carry invariant nondegenerate Hermitian forms constructed as follows.

We extend $F$ to $V^{\otimes d} $ by $F(v_1\otimes ...\otimes v_d,w_1\otimes ...\otimes w_d):=\prod_i{F(v_i,w_i)}$, $d=1,2,...,n$. It is clear that this extension is invariant. Next we need to show that $F$ is nondegenerate on $V^{\wedge d}$. We first show that $F(v,v)\neq 0$ for $v=e_1\wedge ...\wedge e_d$. We have
\begin{eqnarray*}
F(v,v)&=&F(e_1\wedge ...\wedge e_d,e_1\wedge ...\wedge e_d)\\
&=&F(\frac{1}{d!}\sum_{\sigma\in\mathfrak{S}_d}{\mbox{sgn}(\sigma)e_{\sigma(1)}\otimes...\otimes e_{\sigma(d)}},\frac{1}{d!}\sum_{\eta\in\mathfrak{S}_d}{\mbox{sgn}(\eta)e_{\eta(1)}\otimes...\otimes e_{\eta(d)}})\\
&=&\frac{1}{(d!)^2}\sum_{\eta,\sigma\in\mathfrak{S}_d}{\mbox{sgn}(\sigma)\mbox{sgn}(\eta)\prod_i{F(e_{\sigma(i)},e_{\eta(i)})}}\\
&=&\frac{1}{(d!)^2}\sum_{\sigma\in\mathfrak{S}_d}{\mbox{sgn}(\sigma)\mbox{sgn}(\sigma)\prod_i{F(e_{\sigma(i)},e_{\sigma(i)})}}\\
&=&\pm\frac{1}{d!}\neq 0
\end{eqnarray*}
Suppose that $F$ is degenerate on $V^{\wedge d}$, i.e. there exists a vector $w\in V^{\wedge d}$ such that $F(w,u)=0$ for all $u\in V^{\wedge d }$. Using the invariance of $F$ we have $F(Xw,u)+F(w,Xu)=F(Xw,u)=0$ for all $u$. Since $V^{\wedge d}$ is an irreducible representation we can repeat this process to get a basis for $V^{\wedge d}$ consisting of vectors that are orthogonal to all of $V^{\wedge d}$. But his contradicts that $F(v,v)\neq 0$, hence $F$ must be nondegenerate.

Next we want to construct a representation of highest weight $\sum{m_i\omega_i}$. Let $V_d=V^{\wedge d}$ and let $v_d$ denote the highest weight vector of $V_d$.
Consider the (reducible) representation $W=V_1^{\otimes m_1}\otimes ... \otimes V_n^{\otimes m_n}$ and the vector $w=v_1^{\otimes m_1}\otimes ... \otimes v_n^{\otimes m_n}$. The vector $w$ is easily seen to be a weight vector with weight $\sum{m_i\omega_i}$. Take any other weight vector of the form $w'=w_{1,1}\otimes ...\otimes w_{1,m_1}\otimes w_{2,1}\otimes ...\otimes w_{n,m_n}\in W$. Then 
\begin{eqnarray*}
\mbox{weight of}(w')&=&\sum_i\sum_{k=1}^{m_i}{\mbox{weight of}(w_{i,k})}\leq\sum_i\sum_{k=1}^{m_i}{\mbox{weight of}(v_i)}\\&=&\sum_i{m_i\mbox{weight of}(v_i)}=\mbox{weight of}(w)
\end{eqnarray*}
On the other hand, any weight vector is a linear combination of weight vectors of the above form with the same weight.
Thus the weight $\sum{m_i\omega_i}$ is the highest one and the irreducible representation $W'\subset W$ of $\mathfrak{g}$ containing $w$ will be of highest weight $\sum{m_i\omega_i}$. We have that $F(w,w)=\prod_i{F(v_i,v_i)^{m_i}}\neq 0$. By our previous reasoning this means that $F$ is an invariant nondegenerate Hermitian form on $W'$. This concludes the case $\mathfrak{g}=\mathfrak{su}(p,q)$.

Next we treat the case $\mathfrak{g}=\mathfrak{sp}(2n,\mathbb{R})$. The Lie algebra $\mathfrak{sp}(2n,\mathbb{R})$ is defined as the Lie algebra of linear endomorphisms of a real vector space $V\simeq \mathbb{R}^{2n}$ that preserves a nondegenerate skewsymmetric bilinear form $B$.
We can extend $B$ and the action of $\mathfrak{sp}(2n,\mathbb{R})$ to $V^\mathbb{C}=V\otimes_\mathbb{R} \mathbb{C}$. We construct an invariant nondegenerate Hermitian form $F$ on $V^\mathbb{C}$ from $B$ by
$F(v,w):=i B(\bar{v},w)$. We can extend the action of $\mathfrak{sp}(2n,\mathbb{R})$ on $V^\mathbb{C}$ to $\mathfrak{sp}(2n,\mathbb{C})$. This is the \emph{standard representation} of $\mathfrak{sp}(2n,\mathbb{C})$. From this point on the construction is completely analogous to that of $\mathfrak{su}(p,q)$ and is omitted.

For a nonsimple $\mathfrak{g}=\mathfrak{g}_1\oplus ...\oplus \mathfrak{g}_m$ we have that an irreducible representation $V$ of $\mathfrak{g}$ is the tensor product of irreducible representations $V_i$ of the $\mathfrak{g}_i$. We can equip each $V_i$ with a $\mathfrak{g}_i$-invariant Hermitian form $F_i$ as above. We get a $\mathfrak{g}$-invariant form $F$ on $V=V_1\otimes ... \otimes V_m $ by $F(v_1\otimes ... \otimes v_m,w_1\otimes ... \otimes w_m):=\prod_i{F_i(v_i,w_i)}$.

What remains is to show that the form $F$ on an irreducible representation $(V,\rho)$ is unique up to scaling. Suppose that $F_1, F_2$ are two invariant nondegenerate Hermitian forms on $V$. Fixing a basis of $V$ we can represent these forms using matrices as $F_i(v,w)=v^*A_iw$ for some Hermitian matrices $A_i$. That the $F_i$ are invariant translates to 
\begin{eqnarray*}
\rho(X)^*A_1+A_1\rho(X)=0\\
\rho(X)^*A_2+A_2\rho(X)=0
\end{eqnarray*}
Since the $F_i$ are nondegenerate the matrices $A_i$ are invertible and we can rewrite this as 
\begin{eqnarray*}
\rho(X)^*=-A_1\rho(X)A_1^{-1}\\
\rho(X)^*=-A_2\rho(X)A_2^{-1}
\end{eqnarray*}
 We thus get $(A_2^{-1}A_1)\rho(X)= \rho(X)(A_2^{-1}A_1)$,
i.e. $\rho(X)$ commutes with $A_2^{-1}A_1$ for all $X\in\mathfrak{g}$.
This implies that $\rho(X)$ leaves the generalized eigenspaces of $A_2^{-1}A_1$ invariant. Since $\rho$ is irreducible, $A_2^{-1}A_1$ can only have one generalized eigenspace. We can thus express $A_2^{-1}A_1 $ as $A_2^{-1}A_1=\lambda I+N$, where $\lambda\in\mathbb{C}$ and $N$ is some nilpotent matrix. Since $\rho(X)$ commutes with $A_2^{-1}A_1$ it must commute with $N$. This implies that $\mbox{ker}(N)$ is invariant under $\rho(X)$. Since $\rho$ is irreducible this means that $\mbox{ker}(N)$ is either $0$ or $V$. Since by assumption $N$ is a nilpotent matrix we have that $\mbox{ker}(N)=V$ which implies that $N=0$. Thus $A_2^{-1}A_1=\lambda I$ and $A_1=\lambda A_2$. Since $A_1$ and $A_2$ are Hermitian matrices $\lambda$ must be real.
\end{proof}

\begin{proof}[Proof of Theorem \ref{43}]  
Start with an arbitrary decomposition into irreducible invariant subspaces $U=\bigoplus_{i\in I}{U_i}$. There are two possibilities, either $F|_{U_i}$ is degenerate for all $i$ or there exists some $i$ such that it is nondegenerate.
In the second case we claim that $U_i^{\perp}$, the orthogonal complement of $U_i$, is an invariant subspace. Taking $v\in U_i^{\perp}$, we have 
\begin{equation*}
0=F(\rho(X)u,v)+F(u,\rho(X)v)=F(u,\rho(X)v)
\end{equation*}
for all $u\in U_i$, since $U_i$ is invariant. This proves the claim.

We can now take a decomposition of $U_i^{\perp}$ into irreducible subspaces to get a new decomposition of $U=U_i\oplus_F \bigoplus_{j\in I'}{V_j}$. Here $\oplus_F$ denotes a direct sum that is orthogonal with respect to $F$.
Repeating this process we get, with some reindexing, $U=U_1\oplus_F ...\oplus_F U_n \oplus_F \bigoplus_{j\in I'}{V_j}$ where all the $U_i$ are nondegenerate and all the $V_j$ are degenerate irreducible subspaces.

Our next claim is that for an irreducible degenerate space $V_j$ we have $F|_{V_j}\equiv 0$.
Since $F|_{V_j}$ is degenerate there exists a $v\in V_j$ such that $v \perp V_j$. If $\mbox{dim}(V_j)=1$ we are done, if not let $X\in\mathfrak{g}$ be such that $\rho(X)v\not\in \mathbb{C}v$. We have that $0=F(\rho(X)v,u)+F(v,\rho(X)u)=F(\rho(X)v,u)$ for all $u\in V_j$, which means that $\rho(X)v\perp V_j$. Using the fact that $V_j$ is irreducible and repeating this process we can conclude that $V_j\perp V_j$. 
 
Since $F$ is a nondegenerate form there must be some $k\in I'$ such that $F|_{V_j\oplus V_k}\not\equiv 0$.
We claim that $F|_{V_j\oplus V_k}$ is nondegenerate. 
Assume otherwise, then there are vectors $u_j\in V_j$ and $u_k\in V_k$ such that $u_j +u_k \perp V_j\oplus V_k$.
We have $0=F(u_j+u_k,u)=F(u_j,u)$ for all $u\in V_k$ which implies $u_j\perp V_k$. We have further
\begin{equation*}
0=F(\rho(X)u_j,u)+F(u_j,\rho(X)u)=F(\rho(X)u_j,u)
\end{equation*}  
for all $u\in V_k$ which means that $\rho(X)u_j\perp V_k$. Repeating this we get $V_j\perp V_k$ which means $F|_{V_j\oplus V_k}\equiv 0$ contradicting our assumption.

We are left with the case where there are two irreducible invariant degenerate spaces $V,W$ such that $V\oplus W$ is nondegenerate.
We claim that $\mbox{dim} (V)=\mbox{dim} (W)=:n$ and that the signature of $F|_{V\oplus W}$ is $(n,n)$.
This follows from that the maximal dimension of a null subspace of a space of signature $(p,q)$ is $\mbox{min}(p,q)$.
From Lemma \ref{form} we know that there exists an nondegenerate Hermitian form $F_V$ of say signature $(p,q)$ on $V$ invariant under $\rho(\mathfrak{g})$. Take a basis $\{v_i\}_{i=1}^n$ for $V$ orthonormal with respect to $F_V$ with the first $p$ vectors having positive quadratic $F(v_i,v_i)=1$, and let $\{w_j'\}_{j=1}^n$ be some choice of basis for $W$.
Representing $F$ as a matrix with respect to the basis $\{v_1,...,v_n,w_1',...,w_n'\}$ we have $F(u_1,u_2)=u_1^*
\left(\begin{array}{cc}
0&C\\
C^*&0
\end{array}\right) u_2$
for some invertible matrix $C$. Choosing a new basis $\{w_j=C^{-1}w_j'\}_{j=1}^n$ for $W$ we can represent $F$ with the matrix 
$\left(\begin{array}{cc}
0&I\\
I&0
\end{array}\right)$.
An arbitrary element $\rho(X)$ leaves the subspaces $V$ and $W$ invariant.
This means $\rho(X)$ is of the form $\rho(X)=
\left(\begin{array}{cc}
A&0\\
0&B
\end{array}\right)$.
It preserves the form $F$ which implies that $B=-A^*$.
We have further that $A$ preserves the form $F_V$ which means that $A$ is of the form $A=
\left(\begin{array}{cc}
K&Z\\
Z^*&L
\end{array}\right)$ with $K\in M_{p,p}(\mathbb{C}),L\in M_{q,q}(\mathbb{C})$ skewsymmetric and $tr(K+L)=0$.
Now consider the following basis, let $\{e_1,...,e_{2n}\}=\{v_1+w_1,...,v_p+w_p,v_{p+1}-w_{p+1},...,v_n-w_n,v_1-w_1,...,v_p-w_p,v_{p+1}+w_{p+1},...,v_n+w_n\}$. With respect to this basis we have that $\rho(X)=
\left(\begin{array}{cccc}
K&Z&0&0\\
Z^*&L&0&0\\
0&0&K&Z\\
0&0&Z^*&L
\end{array}\right).
$
We see that the subspaces spanned by $\{e_i\}_{i=1}^n$ and $\{e_j\}_{j=n+1}^{2n} $ are invariant and a quick calculation shows that they are nondegenerate.
We have now established an inductive procedure of picking out orthogonal nondegenerate subspaces out of an arbitrary decomposition. This proves the theorem.
 \end{proof}

\renewcommand{\bibname}{{\sc References}}
\begingroup
\let\chapter\section

\endgroup


\begin{bibdiv}
\begin{biblist}

\bib{B9}{article}{
   author={Burger, Marc},
   author={Iozzi, Alessandra},
   author={Wienhard, Anna},
   title={Surface group representations with maximal Toledo invariant},
   journal={Ann. of Math. (2)},
   volume={172},
   date={2010},
   number={1},
   pages={517--566},
   issn={0003-486X},
}

\bib{B8}{article}{
   author={Burger, Marc},
   author={Iozzi, Alessandra},
   author={Wienhard, Anna},
   title={Tight homomorphisms and Hermitian symmetric spaces},
   journal={Geom. Funct. Anal.},
   volume={19},
   date={2009},
   number={3},
   pages={678--721},
   issn={1016-443X},
}

\bib{B11}{article}{
   author={Burger, M.},
   author={Monod, N.},
   title={Continuous bounded cohomology and applications to rigidity theory},
   journal={Geom. Funct. Anal.},
   volume={12},
   date={2002},
   number={2},
   pages={219--280},
   issn={1016-443X},
}

\bib{B14}{article}{
   author={Clerc, Jean-Louis},
   author={{\O}rsted, Bent},
   title={The Gromov norm of the Kaehler class and the Maslov index},
   journal={Asian J. Math.},
   volume={7},
   date={2003},
   number={2},
   pages={269--295},
   issn={1093-6106},
}

\bib{B4}{article}{
   author={Domic, Antun},
   author={Toledo, Domingo},
   title={The Gromov norm of the Kaehler class of symmetric domains},
   journal={Math. Ann.},
   volume={276},
   date={1987},
   number={3},
   pages={425--432},
   issn={0025-5831},
}

\bib{B15}{article}{
   author={Dynkin, E. B.},
   title={Semisimple subalgebras of semisimple Lie algebras},
   language={Russian},
   journal={Mat. Sbornik N.S.},
   volume={30(72)},
   date={1952},
   pages={349--462 (3 plates)},
}

\bib{B16}{book}{
   author={Fulton, William},
   author={Harris, Joe},
   title={Representation theory},
   series={Graduate Texts in Mathematics},
   volume={129},
   note={A first course;
   Readings in Mathematics},
   publisher={Springer-Verlag},
   place={New York},
   date={1991},
   pages={xvi+551},
   isbn={0-387-97527-6},
   isbn={0-387-97495-4},
}

\bib{B12}{book}{
   author={Hall, Brian C.},
   title={Lie groups, Lie algebras, and representations},
   series={Graduate Texts in Mathematics},
   volume={222},
   note={An elementary introduction},
   publisher={Springer-Verlag},
   place={New York},
   date={2003},
   pages={xiv+351},
   isbn={0-387-40122-9},
}

\bib{B5}{article}{
   author={Hamlet, Oskar},
   title={Tight holomorphic maps, a classification},
   journal={J. Lie Theory},
   volume={23},
   date={2013},
   number={3},
   pages={639--654},
   issn={0949-5932},
}



\bib{B3}{book}{
   author={Helgason, Sigurdur},
   title={Differential geometry, Lie groups, and symmetric spaces},
   series={Graduate Studies in Mathematics},
   volume={34},
   note={Corrected reprint of the 1978 original},
   publisher={American Mathematical Society},
   place={Providence, RI},
   date={2001},
   pages={xxvi+641},
   isbn={0-8218-2848-7},
}

\bib{B13}{article}{
   author={Hamlet, Oskar},
   author={Okuda, Takayuki},
   title={Tight maps and holomorphicity, exceptional spaces},
   status={work in progress},
}

\bib{B6}{article}{
   author={Ihara, Shin-Ichiro},
   title={Holomorphic imbeddings of symmetric domains},
   journal={J. Math. Soc. Japan},
   volume={19},
   date={1967},
   pages={261--302},
   issn={0025-5645},
}

\bib{B2}{article}{
   author={Kim, Inkang},
   author={Pansu, Pierre},
   title={Density of Zariski density for surface groups},
   journal={Duke Math. J.},
   volume={163},
   date={2014},
   number={9},
   pages={1737--1794},
   issn={0012-7094},
}


\bib{B10}{book}{
   author={Monod, Nicolas},
   title={Continuous bounded cohomology of locally compact groups},
   series={Lecture Notes in Mathematics},
   volume={1758},
   publisher={Springer-Verlag},
   place={Berlin},
   date={2001},
   pages={x+214},
   isbn={3-540-42054-1},
}

\bib{B7}{article}{
   author={Satake, Ichir{\^o}},
   title={Holomorphic imbeddings of symmetric domains into a Siegel space},
   journal={Amer. J. Math.},
   volume={87},
   date={1965},
   pages={425--461},
   issn={0002-9327},
}

\end{biblist}
\end{bibdiv}
\end{document}